\documentclass[a4paper,12pt]{article}

\usepackage{amsfonts,amssymb,amsmath,dsfont,relsize,accents}
\usepackage{amsthm}
\usepackage[english]{babel}
\usepackage[small,labelfont=it,margin=1em]{caption}
\usepackage{color}
\usepackage[utf8]{inputenc}
\usepackage{latexsym}
\usepackage{mathrsfs}
\usepackage{upgreek}
\usepackage{graphicx}
\usepackage{hyperref}
\usepackage[normalem]{ulem}
\usepackage{float}

\DeclareMathAlphabet{\mathdutchcal}{U}{dutchcal}{m}{n}
\SetMathAlphabet{\mathdutchcal}{bold}{U}{dutchcal}{b}{n}
\DeclareMathAlphabet{\mathdutchbcal}{U}{dutchcal}{b}{n}

\newtheorem{definition}{Definition}[section]
\newtheorem{remark}{Remark}[section]
\newtheorem{theorem}{Theorem}[section]
\newtheorem*{theorem*}{Theorem}
\newtheorem{lemma}{Lemma}[section]
\newtheorem{proposition}[lemma]{Proposition}
\newtheorem{corollary}[lemma]{Corollary}
\theoremstyle{definition}
\newtheorem*{definition*}{Definition}

\newcommand*{\E}{\mathbb{E}}

\newcommand*{\N}{\mathbb{N}}
\renewcommand{\P}{\mathbb{P}}

\newcommand*{\R}{\mathbb{R}}
\newcommand*{\T}{\mathbb{T}}

\renewcommand{\leq}{\leqslant}
\renewcommand{\le}{\leqslant}

\renewcommand{\ge}{\geqslant}
\renewcommand{\subset}{\subseteq}

\hypersetup{colorlinks=true,urlcolor=blue,linkcolor=black,citecolor=black}

\newcommand*{\doi}[1]{\href{http://dx.doi.org/\detokenize{#1}}{doi}}

\renewcommand{\baselinestretch}{1.15}

\makeatletter
\newcommand{\doublewidetilde}[1]{{%
  \mathpalette\double@widetilde{#1}%
}}
\newcommand{\double@widetilde}[2]{%
  \sbox\z@{$\m@th#1\widetilde{#2}$}%
  \ht\z@=.9\ht\z@
  \widetilde{\box\z@}%
}
\makeatother

\begin{document}

\title{The contact process on dynamic regular graphs: monotonicity and subcritical phase}
\author{Bruno Schapira\footnote{Aix-Marseille Universit\'e, \url{bruno.schapira@univ-amu.fr}},
{}
Daniel Valesin\footnote{University of Warwick, \url{daniel.valesin@warwick.ac.uk}}
}
% \date{February 13, 2017}
\maketitle

\begin{abstract}
	We study the contact process on a dynamic random~$d$-regular graph with an edge-switching mechanism, as well as an interacting particle system that arises from the local description of this process, called the herds process. Both these processes were introduced in~\cite{da2021contact}; there it was shown that the herds process  has a phase transition with respect to the infectivity parameter~$\lambda$, depending on the parameter~$\mathsf{v}$ that governs the edge dynamics. Improving on a result of~\cite{da2021contact}, we prove that the critical value of~$\lambda$ is strictly decreasing with~$\mathsf{v}$. We also prove that in the subcritical regime, the extinction time of the herds process started from a single individual has an exponential tail. Finally, we apply these results to study the subcritical regime of the contact process on the dynamic $d$-regular graph. We show that, starting from all vertices infected, the infection goes extinct in a time that is logarithmic in the number of vertices of the graph, with high probability.
\end{abstract}

\section{Introduction}
This paper is a follow-up to~\cite{da2021contact}, which studied the contact process on a dynamic random $d$-regular graph with an edge-flip mechanism introduced in~\cite{CDG07}. The work~\cite{da2021contact} mainly focused on proving the existence of a supercritical regime, where the extinction time of the process 
grows exponentially with the number of vertices of the graph. Here, we show that there is a phase transition between two regimes, where 
the order of magnitude of the extinction time switches abruptly from logarithmic to exponential, as the infection parameter crosses a critical value. The highlight of our analysis is that it allows us to establish that  this critical value of the infection parameter is a strictly monotone function of the rate of the edge-flip mechanism.

\subsection{Contact process on static finite graphs}
The \textbf{contact process} on a graph $G$ is an interacting particle system in which the vertices of the graph can be either  healthy or infected. Healthy vertices get infected at rate $\lambda$ times the number of infected neighbors, where $\lambda>0$ is a fixed parameter of the model, while infected vertices become healthy at rate $1$, independently of each other.
When the graph $G$ is infinite, a quantity of interest is the \textbf{critical rate} $\lambda_c(G)$, defined as the supremum of the values of~$\lambda$ for which the process started from any finite infected set dies out (reaches the all-healthy configuration) almost surely.

In the case when $G$ is finite, the all-healthy configuration is always reached almost surely (regardless of~$\lambda$). The \textbf{extinction time} $\tau_G$ is the hitting time of the all-healthy configuration, for the process started from all infected.
It has been observed in several cases that when $(G_n)_{n\ge 1}$ is a sequence of finite graphs which converges locally to some (rooted) 
infinite graph $G_\infty$, typically the extinction time $\tau_{G_n}$ grows logarithmically with $n$ when $\lambda$ is smaller than $\lambda_c(G_\infty)$ and grows exponentially with $n$ when $\lambda$ is larger than $\lambda_c(G_\infty)$. 
For instance, this has been shown when $G_n$ is a $d$-dimensional cube~$\{0,\dots,n\}^d$~\cite{CGOV84,  DL88, DS88, M93, S85}, a $d$-regular tree up to height $n$~\cite{CMMV14, S01}, or in the case which interests us more here, when $G_n$ is a random $d$-regular graph with $n$ vertices~\cite{LS17, MV16}, in which cases $G_\infty$ is respectively $\mathbb Z^d$, the canopy tree, and the $d$-regular tree $\mathbb T^d$. 

\subsection{Contact process on a dynamical random $d$-regular graph}

We present now the dynamical version of the random $d$-regular graph first introduced and studied 
in~\cite{CDG07}. Throughout the paper, we fix the degree~$d \ge 3$, and whenever we talk about a~$d$-regular graph with~$n$ vertices, we assume that~$nd$ is even. We allow our graphs to contain loops (edges involving the same vertex twice) and parallel edges (multiple edges between the same two vertices), but we will keep writing `graph' instead of some other terminology such as `multi-graph'.

Let~$G$ be a $d$-regular graph with $n$ vertices, and~$e, e'$ be two of its edges; let~$u,v$ be the vertices of~$e$ and~$u',v'$ the vertices of~$e'$. We can define two possible \textbf{switches} of these edges by replacing 
them with either (1) edges with vertices~$\{u,u'\}$ and~$\{v,v'\}$, or (2) edges with vertices~$\{u,v'\}$ and~$\{v,u'\}$. 
We then define a continuous-time Markov chain $(G_t)_{t\ge 0}$ on the space of $d$-regular graphs on a fixed set of $n$ vertices as follows. The initial graph 
$G_0$ is distributed according to the uniform distribution on the set of $d$-regular graphs. 
Then, given the state $G_t$ at time $t$, we prescribe that any of the~$2 \cdot {|E| \choose 2}$ possible edge switches occurs on this graph with rate~$\frac{\mathsf{v}}{nd}$, where~$\mathsf{v}>0$ is a positive parameter. It is readily 
seen that the uniform distribution on random $d$-regular graphs is stationary with respect to this dynamics, 
and moreover, that any fixed edge is involved in a switch at a rate which converges to $\mathsf v$, as $n\to \infty$.

We next consider the process $(G_t,\xi_t)_{t\ge 0}$ where $(G_t)_{t\ge 0}$ is as above, and $(\xi_t)_{t\ge 0}$ is a contact process evolving on the dynamic graph. As previously mentioned, the process starts from the configuration where all vertices are infected, and our main interest is in the time $\tau_{(G_t)}$ 
when the process reaches the all-healthy configuration. The following result was proved in~\cite{da2021contact}.   In both this theorem and in Theorem~\ref{theo.1} below, the probability measure~$\mathbb{P}$ includes the randomness of both the random dynamic graph and the contact process.
\begin{theorem}[\cite{da2021contact}] 
For each $\mathsf v>0$, there exists $\bar \lambda(\mathsf v) \in (0,\lambda_c(\mathbb T_d))$, such that the following holds. 
For any $\lambda>\bar \lambda(\mathsf v)$, there exists $c>0$ such that 
$$\mathbb P\big(\tau_{(G_t)}> \exp\{cn\}\big) \xrightarrow{n\to \infty} 1. $$ 
\end{theorem}

Note in particular the interesting feature that~$\bar{\lambda}(\mathsf v)$ is strictly smaller than $\lambda_c(\mathbb T_d)$, which means that 
the dynamics of the graph helps the contact process to survive for a longer time than in the static model.

\subsection{Main results}
In this paper, we complete the picture by proving the following result. 
\begin{theorem}\label{theo.1}
For each $\mathsf v>0$, there exists $\bar \lambda(\mathsf v) \in (0,\lambda_c(\mathbb T_d))$ such that the following holds. 
\begin{itemize}
\item[(i)] For any $\lambda>\bar \lambda(\mathsf v)$, there exists $c>0$ such that 
$$\mathbb P\big(\tau_{(G_t)}> \exp\{cn\}\big) \xrightarrow{n\to \infty} 1. $$ 
\item[(ii)] For any $\lambda<\bar \lambda(\mathsf v)$, there exists $C>0$ such that 
$$\mathbb P\big(\tau_{(G_t)}> C\log n\big) \xrightarrow{n\to \infty} 0. $$
\end{itemize} 
\end{theorem} 
As in the static case, the value $\bar \lambda(\mathsf v)$ corresponds to the critical value for the contact process on a limiting model, which in our case is called the \textbf{herds process}. 
This was introduced and analyzed in~\cite{da2021contact}, where in particular a phase transition delimited by a positive and finite parameter~$\bar{\lambda}(\mathsf{v})$ was established.
 Here we improve upon this result by showing that the herds process exhibits a form of sharp threshold phenomenon, namely that 
the tail distribution of the extinction time decays exponentially fast in the whole subcritical regime (see Lemma~\ref{lem_revineq} below). An informal description of the herds process is given in the next subsection, and a precise definition is given in Section~\ref{s_herds}.

Our second result answers a question of~\cite{da2021contact} concerning the monotonicity of $\bar \lambda(\mathsf v)$. 
\begin{theorem}\label{theo.2}
The mapping $\mathsf v\mapsto \bar \lambda(\mathsf v)$ is strictly decreasing. 
\end{theorem}

\subsection{Methods of proof and organization of the paper}
As in~\cite{da2021contact}, the proof of Theorem~\ref{theo.1} relies on a detailed analysis of the herds process. 
Informally, this process evolves as a contact process on a family of $d$-regular trees, where the number of trees also evolves with time. On each tree, the process obeys the same rules as the usual contact process, regarding infection and recovery (though we adopt a slight change of terminology: the vertex states `healthy' and `infected' here are called `empty' and `occupied by a particle', respectively). In addition, each edge in any of the existing tree splits the tree into two pieces at a constant rate $\mathsf v$. When this happens, the two disjoint pieces of the tree are completed to form two new copies of a $d$-regular tree. 

The value $\bar \lambda(\mathsf v)$ is defined as the threshold for the infection parameter $\lambda$ above which the process has a positive probability of surviving forever, when starting from a single tree with a single particle. The heart of the proof is to show that when $\lambda$ is smaller than this threshold, 
the probability to survive for a time larger than $t$ decays exponentially fast with $t$. This is obtained using 
some coupling argument with a two-type herds process, which allows to show that the expected number of infected particles at time $t$, denoted (in this section only) $F(\lambda,\mathsf v,t)$, is a sub-multiplicative sequence (as a function of the time parameter), see Section~\ref{lem_sub_multiplicative}. Using this, we can define the rate of exponential decay of this function, $\varphi(\lambda,\mathsf v)$, and then the proof boils down to showing that it is strictly increasing with respect to both parameters. This is obtained via a kind of Russo's formula, see Proposition~\ref{prop.derivative}, which in our setting is quite involved, compared to the original formula from percolation theory. 
Moreover, the strict monotonicity of~$\varphi(\lambda,\mathsf v)$ also proves Theorem~\ref{theo.2}.

Another important ingredient is to control the higher moments of the number of infected particles, which requires some serious additional technical work, due to the non-linearity of these functionals, see Section~\ref{s_proof_derivative}. 
Bounding higher moments is needed to control the total number of particle births (or in the usual contact process terminology, infections) up to the extinction time, which 
in turn allows one to couple the contact process on a large dynamic $d$-regular graph with the herds process, up to this extinction time. This coupling argument is explained in Section~\ref{s_finite_graph}, where we complete the proof of Theorem~\ref{theo.1}.

\subsection{Related works}\label{ss_refs}
The study of the phase transition for long vs. short time extinction for the contact process on finite graphs, 
and the closely related question 
of metastability in the supercritical regime, has been studied intensively over the past  years. 
In particular as we already mentioned on finite boxes of $\mathbb Z^d$ \cite{CGOV84,DL88,DS88,DST89,M93,S85}, finite $d$-regular trees \cite{CMMV14,S01}, and random $d$-regular graphs~\cite{LS17,MV16}, but also in a number of other examples, such as the configuration model with power law degree distribution~\cite{CS15,CD09,MVY13},
or with general degree-distribution~\cite{BNNS21, HD20}, preferential attachment graphs~\cite{BBCS05,C17}, Erd\H{o}s-R\'enyi random graphs~\cite{BNNS21}, inhomogeneous random graphs~\cite{C19}, or hyperbolic random graphs~\cite{LMSV21}. There are also 
results concerning some general classes of finite graph sequences~\cite{MMVY16,SV17}.

On the other hand, the study of the contact process on dynamical graphs started more recently, see e.g.~\cite{da2021contact,LJ23,HUVV21,JLM19,JLM22,JM17,LR20,remenik08,seiler2022contact}.

\section{Preliminaries on the herds process}\label{s_herds}
In this section, we give a formal definition of the herds process and introduce notation. We also present some tools that will be employed in the analysis of this process in later sections, namely, stochastic domination by a pure-birth process and a submultiplicativity inequality for the expectation of the number of particles.

\subsection{Definition and construction}
Throughout this paper, we fix~$d \ge 3$ and let~$\T^d$ denote the infinite~$d$-regular {rooted} tree. {The root is denoted by $o$, and we write $u\sim v$ when two vertices $u$ and $v$ are neighbors.}
%We endow the vertices of~$\T^d$ with an arbitrary total order.

\begin{definition} \label{def_shapes}
Let
	\begin{equation*}
		P_\mathsf{f}(\T^d) := \{A \subset \T^d:\; A \text{ is finite and non-empty}\}.
	\end{equation*}
We call each~$A \in P_\mathsf{f}(\mathbb{T}^d)$ a \textbf{herd shape}, and each~$x \in A$ a \textbf{particle} of~$A$. 
	Given a herd shape~$A$ and an edge~$e=\{u,v\}$ of~$\T^d$ with~$u$ {closer to the root}  than~$v$, {in the graph distance of~$\T^d$,} define
	\[A_{e,1}:= \{w \in A:\; w \text{ is closer to $u$ than to $v$}\},\quad A_{e,2}:=A \backslash A_{e,1}.\]
	We say that~$e$ is an \textbf{active edge of} $A$ if~$A_{e,1} \neq \varnothing$ and~$A_{e,2} \neq \varnothing$.
\end{definition}

\begin{definition} \label{def_configurations}
Define the \textbf{set of herd configurations}
\[\mathcal{S} :=\{\xi:P_\mathsf{f}(\mathbb{T}^d) \to \mathbb{N}_0 \text{ with } \textstyle{\sum_{A}} \xi(A) < \infty\}.\] 
	In a herd configuration~$\xi \in \mathcal{S}$,~$\xi(A)$ is interpreted as the number of herds with shape~$A$. Given~$A \in P_\mathsf{f}(\T^d)$, we let~$\delta_A$ denote the herd configuration such that~$\delta_A(B) = 1$ if~$B = A$, and~$\delta_A(B) = 0$ otherwise.
	An \textbf{enumeration} of~$\xi \in \mathcal{S}$ is a sequence~$A_1,\ldots,A_m \in P_\mathsf{f}(\T^d)$ such that~$\xi = \sum_{i=1}^m \delta_{A_i}$.
\end{definition}
We will generally denote deterministic elements of~$\mathcal{S}$ by the letter~$\xi$, and random elements of~$\mathcal{S}$ by~$\Xi$.

\begin{definition}\label{def_herds_process}
	The \textbf{herds process}~$(\Xi_t)_{t \ge 0}$ with birth rate~$\lambda > 0$ and splitting rate~$\mathsf{v} > 0$ is a continuous-time Markov chain on~$\mathcal{S}$ whose dynamics is given by the following description of possible jumps and corresponding rates:
\begin{itemize}
	\item[(a)] \emph{death in herds with more than one particle:} for each~$A$ with~$|A| > 1$ such that~$\xi(A) > 0$, and for each~$x \in A$, with rate~$\xi(A)$, the process jumps from~$\xi$ to~$\xi - \delta_A + \delta_{A\backslash \{x\}}$;
	\item[(b)] \emph{death in herds with one particle:} for each~$A$ with~$|A|=1$ such that~$\xi(A) > 0$, with rate~$\xi(A)$, the process jumps from~$\xi$ to~$\xi - \delta_A$;
	\item[(c)] \emph{birth:} for each~$A$ such that~$\xi(A) > 0$, and for each~$y \in \mathbb{T}^d$ with~$y \notin A$, with rate~$\lambda \cdot |\{x\in A:x \sim y\}|\cdot \xi(A)$, the process jumps from~$\xi$ to~$\xi - \delta_A + \delta_{A \cup \{y\}}$;
	\item[(d)] \emph{split:} for each~$A$ such that~$\xi(A) > 0$, and for each active edge~$e$ of~$A$, with rate~$\mathsf{v}\cdot \xi(A)$, the process jumps from~$\xi$ to~$\xi - \delta_A + \delta_{A_{e,1}} + \delta_{A_{e,2}}$.
\end{itemize}
\end{definition}

\begin{remark}
\textit{A priori}, it could be the case that the above description gave rise to an explosive chain, that is, a chain that jumps infinitely many times in a bounded time interval. So, strictly speaking, the process is only defined up to the explosion time (the infimum of times~$t > 0$ such that infinitely many jumps happen in~$(0,t)$). However, we will show shortly (see Corollary~\ref{cor_non_exp} below) that the chain is in fact not explosive.
\end{remark}

\begin{remark}\label{rem.statespace}
	Our choice for the state space~$\mathcal{S}$ of the herds process makes it so that, in case there are multiple herds with the same shape~$A$ (meaning that~$\xi(A) \ge 2$), then these herds are indistinguishable. An alternative choice was made in~\cite{da2021contact}: there, a state of the process was an index set~$\mathcal{J}$ and a mapping from~$\mathcal{J}$ to~$P_{\mathsf{f}}(\T^d)$, so that each~$i \in \mathcal{J}$ represented a different herd. This alternative choice requires heavier notation, but has some advantages; for instance, when a particle is born in a herd, it makes sense to consider the herd before and after the birth (since it keeps the same index). 
 Although here we will adopt the leaner description of Definition~\ref{def_herds_process}, we will sometimes pretend that a richer description is available. For instance, in one of our arguments {(see Lemma~\ref{lem_only_one_sep})} we fix a particle in a herd at time~$0$, and consider the evolution of the cardinality of the herd containing that particle for times~$t \ge 0$.
\end{remark}

We let~$\mathbb{P}$ be a probability measure under which the herds process is defined, and~$\mathbb{E}$ the associated expectation. When we want to be explicit about the parameters, we will write~$\mathbb{P}_{\lambda,\mathsf{v}}$ and~$\mathbb{E}_{\lambda,\mathsf{v}}$. \emph{When no explicit mention regarding the initial configuration is made, we assume it to consist of a single herd with a single particle {placed at the root vertex}}.  We may write~{$\mathbb{P}_{\lambda,\mathsf{v}}(\cdot \mid \Xi_0= \xi)$} (and similarly~{$\E_{\lambda,\mathsf{v}}[\cdot\mid \Xi_0=\xi]$}) to specify some other initial configuration~$\xi \in \mathcal{S}$. 

\begin{definition} 
Given~$\xi \in \mathcal{S}$, we let
	\begin{equation} \label{eq_def_X}
		X(\xi) := \sum_{A \in P_\mathsf{f}(\T^d)} |A|\cdot \xi(A),
	\end{equation}
	where~$|\cdot|$ denotes cardinality; that is,~$X(\xi)$ is the total number of particles among all herds in~$\xi$. We also let
	\begin{equation}\label{eq_def_cale}
		\mathscr{E}(\xi) := \sum_{A \in P_\mathsf{f}(\T^d)} |\{\text{active edges of $A$}\}| \cdot \xi(A),
	\end{equation}
	the total number of active edges among all herds of~$\xi$.
For the herds process~$(\Xi_t)_{t \ge 0}$ (started from an arbitrary configuration), we write
\begin{equation}
	\label{eq_def_Xt}
X_t:= X(\Xi_t), \qquad \mathscr{E}_t:= \mathscr{E}(\Xi_t),
\end{equation}
with~$X$ and~$\mathscr{E}$ as in~\eqref{eq_def_X} and~\eqref{eq_def_cale}, respectively.

\end{definition}

We now state a useful stochastic domination result. The proof involves a quick inspection of transition rates, and we omit it.

\begin{lemma}[Domination by pure-birth chain]\label{lem_explosion}
	For the herds process~$(\Xi_t)_{t \ge 0}$ with parameters~$\lambda$,~$\mathsf{v}$ and some initial configuration~$\xi \in \mathcal{S}$, let~$N_t$ denote the number of birth events until time~$t$. Let~$(Z_t)_{t \ge 0}$ be the continuous-time Markov chain on~$\N$ with~$Z_0 = X(\xi)$ and jump rates
	\begin{equation}
		\label{eq_pb_rates}
	q(i,i+1)=d\lambda i, \qquad q(i,j) = 0\text{ for }j \neq i+1.
	\end{equation}
	Then,~$(N_t)_{t \ge 0}$ is stochastically dominated by~$(Z_t - Z_0)_{t \ge 0}$. In particular,~$\P(N_t < \infty) \ge \P(Z_t < \infty) = 1$ for any~$t$.
\end{lemma}
\begin{corollary}\label{cor_non_exp}
	The herds process is non-explosive.
\end{corollary}
\begin{proof}
	When there are finitely many birth events in~$[0,t]$, there are also finitely many death events and split events in~$[0,t]$, so the herds process performs finitely many jumps of any kind in~$[0,t]$.
\end{proof}

We also have the following important consequence concerning the processes~$(X_t)$ and~$(\mathscr{E}_t)$ defined in~\eqref{eq_def_Xt}.
\begin{corollary}
	\label{cor_explosion}
	For any~$\lambda > 0$,~$\mathsf{v} >0$,~$T > 0$ and~$k \ge 1$, there exists~$c > 0$ such that the herds process~$(\Xi_t)_{t \ge 0}$ with parameters~$\lambda,~\mathsf{v}$ and arbitrary (deterministic) initial configuration~$\xi$ satisfies
	\begin{equation*}
		\E\left[\left.\max_{0 \le t \le T} (X_t)^k \;\right|\; \Xi_0 = \xi\right] \le cX(\xi)^k
	\end{equation*}
	and
	\begin{equation*}
\E\left[\left.\max_{0 \le t \le T} (\mathscr{E}_t)^k \;\right|\; \Xi_0 = \xi\right] \le c(X(\xi)+\mathscr{E}(\xi))^k.
	\end{equation*}
\end{corollary}
\begin{proof}
	Again let~$N_t$ denote the number of births in the herds process until time~$t$. Note that
	\begin{equation}\label{eq_to_justify_xe}
\max_{0 \le t \le T} X_t \le X_0 + N_T \qquad \text{and} \quad \max_{0 \le t \le T} \mathscr{E}_t \le \mathscr{E}_0 + N_T;
	\end{equation}
	to justify the latter, we observe that deaths and splits can only decrease~$\mathscr{E}_t$, while a birth can increase~$\mathscr{E}_t$ by at most one.

	Let~$(Z_t)_{t \ge 0}$ be the pure-birth chain of Lemma~\ref{lem_explosion}, started with~$Z_0 = X_0$. By that lemma,~$(Z_t-Z_0)_{t \ge 0}$ stochastically dominates~$(N_t)_{t \ge 0}$. Then,
	\[\E\left[\left.\max_{0 \le t \le T} (X_t)^k\;\right|\; \Xi_0 = \xi\right] \le \E \left[\left.(X_0 + N_T)^k\;\right|\; \Xi_0 = \xi \right] \le \E[(X_0  + Z_T-Z_0))^k],\]
	where the last expectation is with respect to the probability measure under which~$(Z_t)$ is defined (note that~$X_0$  is fixed and deterministic). Next, the law of~$Z_T - Z_0$ is equal to the law of~$\sum_{i=1}^{Z_0} \zeta^{(i)}_T$, where~$(\zeta^{(1)}_t),\ldots,(\zeta^{(Z_0)}_t)$ are independent pure-birth processes with rates as in~\eqref{eq_pb_rates}, each started with a population of one.
	Using Minkowski's inequality,
	\begin{align*}\E[(X_0+ Z_T-Z_0)^k]^{1/k} \le X_0+ \sum_{i=1}^{Z_0} \E[(\zeta^{(i)}_T)^k]^{1/k} = cX_0,
\end{align*}
	where~$c:= \E[(\zeta^{(i)}_T)^k]^{1/k}+1$, {which by~\cite[Corollary 1 p.111]{athreya}} is finite and only depends on~$\lambda$,~$k$ and~$T$. This completes the proof of the first bound. For the second one, we start using the second bound in~\eqref{eq_to_justify_xe}:
\begin{align*}
	\E\left[\left.\max_{0 \le t \le T} (\mathscr{E}_t)^k\;\right|\; \Xi_0 = \xi\right] &\le \E \left[\left.(\mathscr{E}_0 + N_T)^k\;\right|\; \Xi_0 = \xi \right] \\ &= \E \left[\left.( X_0 +\mathscr{E}_0 + N_T - X_0)^k\;\right|\; \Xi_0 = \xi \right], 
\end{align*}
and then complete the proof as in the previous case.
\end{proof}

 We now  present three properties of the herds process (in Lemmas~\ref{lem_auto}, \ref{lem_decompose} and~\ref{lem_dominate} below).  In all three cases, the proof is elementary and omitted.

\begin{lemma}[Invariance under tree automorphisms]\label{lem_auto}
Let~$\psi: \mathbb{T}^d \to \mathbb{T}^d$ be a graph automorphism. Fix~$A \in P_\mathsf{f}(\mathbb{T}^d)$ and let~$(\Xi_t)_{t \ge 0}$ and~$(\Xi_t')_{t \ge 0}$ be herds processes started from~$\delta_A$ and~$\delta_{\psi(A)}$, respectively. Then, the process
\[\sum_{A \in P_\mathsf{f}(\mathbb{T}^d)} \Xi_t(A) \cdot \delta_{\psi(A)},\quad t \ge 0\]
has the same distribution as~$(\Xi_t')_{t \ge 0}$. In particular, the processes~$(X(\Xi_t))_{t \ge 0}$ and~$(X(\Xi'_t))_{t \ge 0}$ have the same distribution.
\end{lemma}

\begin{lemma}[Decomposition into independent processes]\label{lem_decompose}
	Let~$\xi \in \mathcal{S}$ with enumeration~$\xi = \sum_{i=1}^n \delta_{A_i}$, where~$A_1,\ldots,A_n \in P_\mathsf{f}(\T^d)$. Then, the herds process started from $\xi$ has the same distribution as~$(\Xi^{(1)}_t + \cdots + \Xi^{(n)}_t)_{t \ge 0}$,
	where~$(\Xi^{(1)}_t)_{t\ge 0}$,~$\ldots$,~$(\Xi^{(n)}_t)_{t\ge 0}$ are independent herds processes, started from $\delta_{A_1}$,~$\ldots$,~$\delta_{A_n}$, respectively.
\end{lemma}

Before stating the third property, we define a partial order on~$\mathcal{S}$.

\begin{definition}\label{def.stoch.dom}
	Given two herd configurations~$\xi$ and~$\xi'$, we write~$\xi \preceq \xi'$ if there exist enumerations
	\[\xi= \sum_{i=1}^{m} \delta_{A_i},\qquad \xi' = \sum_{j=1}^{n} \delta_{A_j'}\]
	such that~$m \le n$ and~$A_i \subset A_i'$ for~$i = 1,\ldots,m$.
\end{definition}

\begin{lemma}[Attractiveness]\label{lem_dominate}
If~$\xi \preceq \xi'$, then~$(\Xi_t)_{t \ge 0}$ started from~$\xi$ is stochastically dominated (with respect to~$\preceq$) by~$(\Xi'_t)_{t \ge 0}$ started from~$\xi'$.
\end{lemma}

We now give a definition pertaining to extinction vs. survival of the herds process, and introduce the value~$\bar{\lambda}(\mathsf{v})$ that appears in the statements of our main theorems.

\begin{definition}\label{def_critical_value}
    We say that the herds process \textbf{survives} if the event
    $\{\sum_A \Xi_t(A) > 0 \;\forall t\}$ 
    occurs, that is, the process always has particles; otherwise we say that the process \textbf{dies out}. For any~$\mathsf{v}> 0$, we define~$\bar{\lambda}(\mathsf{v})$ as the supremum of the values of~$\lambda$ for which the process with parameters~$\lambda$ and~$\mathsf{v}$ dies out with probability 1.
\end{definition}

It is easy to see that~$\bar{\lambda}(\mathsf{v}) \ge 1/d$. Indeed, when~$\lambda < 1/d$, the rate at which existing particles die always exceeds the rate at which new particles are born, so the process eventually reaches the empty configuration. A moment's thought, using for instance a comparison with a branching process, shows that~$\bar{\lambda}(\mathsf{v}) < \infty$.

\subsection{Sub-multiplicativity of number of particles} \label{lem_sub_multiplicative}
The goal of this section is to prove the following inequality. Recall that~$X_t$ denotes the number of particles in the herds process at time~$t$. Also recall that, whenever the initial condition of the herds process is omitted (say, as in the expectation in the right-hand side of~\eqref{eq_subadd0} below), it is equal to~$\delta_{\{o\}}$.

\begin{proposition}\label{prop_subad}
For any~$t \ge 0$ and~$\xi \in \mathcal{S}$ we have
\begin{equation}\label{eq_subadd0}\mathbb{E}[X_t^p \mid \Xi_0 = \xi] \le X(\xi)^p \cdot \mathbb{E}[X_t^p ].\end{equation}
Consequently, for any~$t,s \ge 0$,~$p \ge 1$ and~$\xi \in \mathcal{S}$, we have
	\label{eq_subad}
\begin{equation}\label{eq_subadd}
	\E[X_{t+s}^p\mid \Xi_0 = \xi] \le \E[X_s^p\mid \Xi_0 = \xi] \cdot \E[X_t^p].
\end{equation}
\end{proposition}

This proposition will be a consequence of the following lemma.
\begin{lemma}\label{lem_general_p}
Let~$A,B \in P_\mathsf{f}(\mathbb{T}^d)$ be disjoint and~$p \ge 1$. Then,
\begin{equation*}
\mathbb{E}[X_t^p \mid \Xi_0 = \delta_{A \cup B}]^{1/p} \le \mathbb{E}[X_t^p \mid \Xi_0 = \delta_A]^{1/p}+ \mathbb{E}[X_t^p \mid \Xi_0 = \delta_B]^{1/p}.
\end{equation*}
\end{lemma}

We postpone the proof of this lemma, and for now show how it implies the proposition:

\begin{proof}[Proof of Proposition~\ref{prop_subad}]
We claim that for any~$p \ge 1$,~$A \in P_\mathsf{f}(\mathbb{T}^d)$, and~$t \ge 0$,
\begin{equation}\label{eq_new_claim}
\mathbb{E}[X_t^p \mid \Xi_0 = \delta_A]^{1/p} \le |A| \cdot \mathbb{E}[X_t^p ]^{1/p}.
\end{equation}
We prove this by induction on~$|A|$. For~$|A| = 1$ the above holds with an equality, by Lemma~\ref{lem_auto}. For the induction step, we assume that~$|A| \ge 2$, take~$u \in A$ and bound, using Lemma~\ref{lem_general_p}, Lemma~\ref{lem_auto} and the induction hypothesis:
\begin{align*}
\mathbb{E}[X_t^p \mid \Xi_0 = \delta_A]^{1/p} &\le  \mathbb{E}[X_t^p \mid \Xi_0 = \delta_{A \backslash \{u\}}]^{1/p} + \mathbb{E}[X_t^p \mid \Xi_0 = \delta_{\{u\}}]^{1/p} \\
&\le (|A\backslash \{u\}|)\cdot \mathbb{E}[X_t^p]^{1/p} + \mathbb{E}[X_t^p ]^{1/p}\\
&= |A| \cdot \mathbb{E}[X_t^p ]^{1/p}.
\end{align*}

Now take~$\xi \in \mathcal{S}$ with enumeration~$\xi = \sum_{i=1}^m \delta_{A_i}$. Let~$(\Xi^{(1)}_t),\ldots,(\Xi^{(m)}_t)$ be independent herds processes, started from~$\delta_{A_1},\ldots,\delta_{A_m}$, respectively. By Lemma~\ref{lem_decompose}, we have
\begin{align*}
\mathbb{E}[X_t^p \mid \Xi_0 = \xi]^{1/p} &= \mathbb{E}\left[\left( \sum_{i=1}^m X(\Xi^{(i)}_t)\right)^p\right]^{1/p}.
\end{align*}
Minkowski's inequality gives
\begin{equation}\label{eq_fragment_one}
\mathbb{E}\left[\left( \sum_{i=1}^m X(\Xi^{(i)}_t)\right)^p\right]^{1/p} \le \sum_{i=1}^m \mathbb{E}\left[ X(\Xi^{(i)}_t)^p \right]^{1/p}. 
\end{equation}
By~\eqref{eq_new_claim}, the right-hand side is smaller than
\begin{align*}
\sum_{i=1}^m|A_i|\cdot \mathbb{E}\left[X_t^p \right]^{1/p} = X(\xi) \cdot \mathbb{E}\left[X_t^p \right]^{1/p}.
\end{align*}
We have thus proved~\eqref{eq_subadd0}. To prove~\eqref{eq_subadd}, we use the Markov property. Let~$s,t \ge 0$ and~$\xi \in \mathcal{S}$; we have:
\begin{align*}
\mathbb{E}[X_{t+s}^p \mid \Xi_0 = \xi] &= \sum_{\xi' \in \mathcal{S}} \mathbb{E}[X_{t+s}^p \mid \Xi_s = \xi']\cdot \mathbb{P}(\Xi_s = \xi' \mid \Xi_0 = \xi) \\
&\le \sum_{\xi' \in \mathcal{S}} X(\xi')^p \cdot \mathbb{E}[X_t^p ] \cdot \mathbb{P}(\Xi_s = \xi' \mid \Xi_0 = \xi)\\
&= \E[X_s^p\mid \Xi_0 = \xi] \cdot \E[X_t^p].
\end{align*}
\end{proof}

To prove Lemma~\ref{lem_general_p}, we will define an auxiliary process, {which informally describes two herds processes that evolve almost independently, except that they share the same splitting events.} We start by defining the state space of this process.

\begin{definition}
Let
\[P_{\mathsf{f},2}(\mathbb{T}^d) := \{(A,B):\;A,B \subset \mathbb{T}^d,\; A \cup B \text{ finite and non-empty}\}\]
and
\[\mathcal{S}_2 := \{\widetilde{\xi}:P_{\mathsf{f},2}(\T^d) \to \N_0\text{ with } \textstyle{\sum_{(A,B)}} \widetilde{\xi}(A,B) < \infty\}.\]
\end{definition}
We interpret an element~$(A,B) \in P_{\mathsf{f},2}(\mathbb{T}^d)$ as a \textbf{two-type herd}, that is, there are two species of particles, one of which occupies~$A$ and the other~$B$. We emphasize that~$A$ and~$B$ need not be disjoint, and one of them, but not both, can be empty. We will also need some projection functions from~${\mathcal{S}}_2$ to~$\mathcal{S}$.

\begin{definition}
We define~$\pi,\pi_1,\pi_2: {\mathcal{S}}_2 \to \mathcal{S}$ by setting, for~$\xi' = \sum_{i=1}^m \delta_{(A_i,B_i)}$:
\begin{equation*}
	\pi(\widetilde{\xi}) = \sum_{i=1}^m \delta_{A_i \cup B_i},\quad \pi_1(\widetilde{\xi}) = \sum_{i:A_i \neq \varnothing} \delta_{A_i}, \quad \pi_2(\widetilde{\xi}) = \sum_{i: B_i \neq \varnothing} \delta_{B_i}.
\end{equation*}
\end{definition}

This definition is illustrated in Figure~\ref{fig_drawing}.

\begin{figure}[htb]
\begin{center}
\setlength\fboxsep{0cm}
\setlength\fboxrule{0cm}
\fbox{\includegraphics[width=\textwidth]{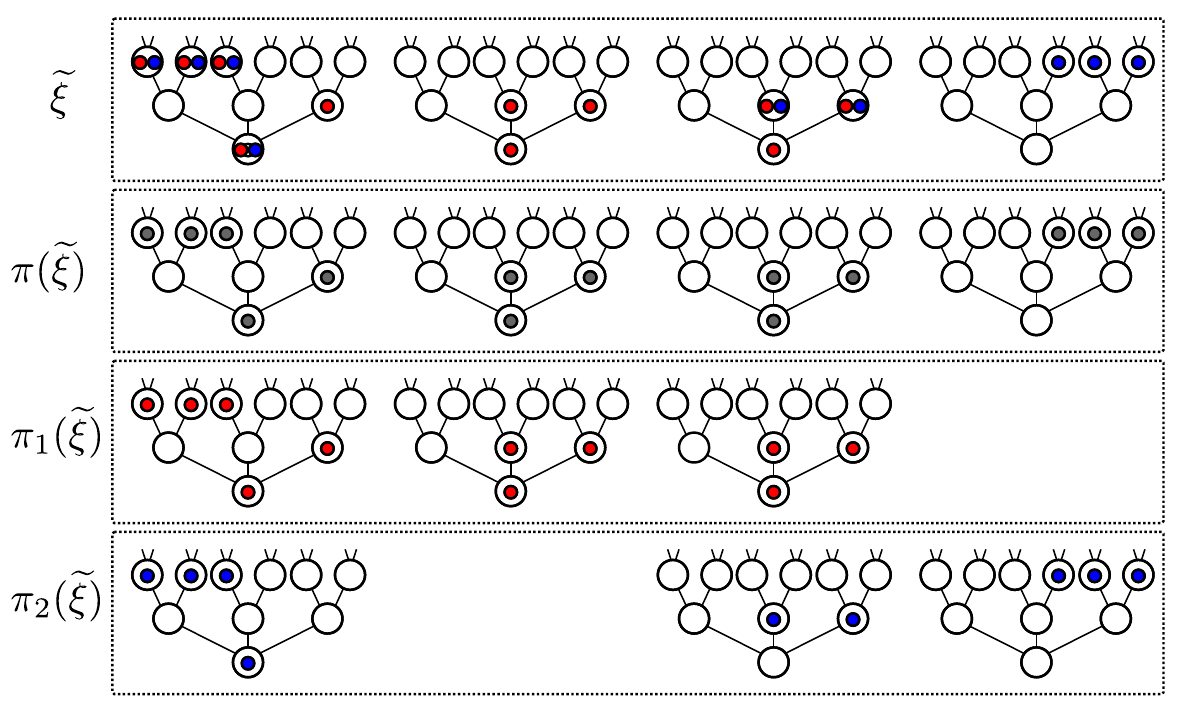}}
\end{center}
	\caption{\label{fig_drawing} {A two-type herds configuration~$\widetilde{\xi}$ is depicted on top, with the two types represented in blue and red (it is assumed that in the four herds of this configuration, no particle is present other than the ones depicted). The projections~$\pi(\widetilde{\xi})$,~$\pi_1(\widetilde{\xi})$ and~$\pi_2(\widetilde{\xi})$ are shown in the second, third and fourth rows, respectively.}}
\end{figure}

\begin{definition}
    The \textbf{two-type herds process}~$(\widetilde{\Xi}_t)_{t \ge 0}$ (with rates~$\lambda$ and~$\mathsf{v}$) is a continuous-time Markov chain on~${\mathcal{S}}_2$ with transitions described as follows. At any time~$t\ge 0$, and for any~$(A,B)$ with $\widetilde{\Xi}_t(A,B)\ge 1$, both $A$ and $B$ are subject, independently of each other, to death and birth mechanisms as in the original herds process (in particular, if either of them is empty, it stays empty). However, they are subject together to the same splitting events: an edge $e$ is said to be active if~$(A\cup B)_{e,1} \neq \varnothing$ and $(A\cup B)_{e,2} \neq \varnothing$. Then, a split occurs at any active edge $e$ with rate $\mathsf{v}$, and when this happens, $(A,B)$ is split into the two pairs $(A_{e,1},B_{e,1})$ and $(A_{e,2},B_{e,2})$. (In case~$A = \varnothing$, we let~$A_{e,1} = A_{e,2} = \varnothing$, and similarly for~$B$).   
\end{definition}

We record two observations about the two-type herds process in the following lemma. The proof is done by comparing jump rates, and we omit it.

\begin{lemma}\label{lem_projections} Let~$(A,B) \in P_\mathsf{f}(\T^d) \times P_\mathsf{f}(\T^d)$, and let~$(\widetilde{\Xi})_{t \ge 0}$ be the two-type herds process started from~$\delta_{(A,B)}$.
	\begin{itemize}
		\item[(a)]  The processes~$(\pi_1(\widetilde{\Xi}_t))_{t \ge 0}$ and~$(\pi_2(\widetilde{\Xi}_t))_{t \ge 0}$ are herds processes started from~$\delta_A$ and~$\delta_B$, respectively.
		\item[(b)]   The herds process started from~$\delta_{A \cup B}$ is stochastically dominated (in the sense of the partial order~$\preceq$) by~$(\pi(\widetilde{\Xi}_t))_{t \ge 0}$.
		\end{itemize}
\end{lemma}

\begin{proof}[Proof of Lemma~\ref{lem_general_p}]
Fix~$p \ge 1$ and disjoint sets~$A,B \in P_\mathsf{f}(\mathbb{T}^d)$. Let~$(\widetilde{\Xi}_t)_{t \ge 0}$ denote a two-type herds process started from~$\delta_{(A,B)}$, defined under a probability measure~$\widetilde{\P}$ (with expectation operator~$\widetilde{\E}$). By Lemma~\ref{lem_projections}(b), and the fact that~$X(\cdot)$ is monotone with respect to the partial order~$\preceq$, we have
	\[\E[X_t^p \mid \Xi_0 = \delta_{A \cup B}] \le \widetilde{\E}[X(\pi(\widetilde{\Xi}_t))^p].\]
Next, noting that~$X(\pi(\widetilde{\Xi}_t))  \le X(\pi_1(\widetilde{\Xi}_t)) + X(\pi_2(\widetilde{\Xi}_t))$,
we have
\[\widetilde{\E}[X(\pi(\widetilde{\Xi}_t))^p] \le \widetilde{\E}[(X(\pi_1(\widetilde{\Xi}_t)) + X(\pi_2(\widetilde{\Xi}_t))^p].\]
Putting these inequalities together (raised to the power~$1/p$) and using Minkowski's inequality, we obtain
\[\E[X_t^p \mid \Xi_0 = \delta_{A \cup B}]^{1/p} \le  \widetilde{\E}[X(\pi_1(\widetilde{\Xi}_t))^p]^{1/p} + \widetilde{\E}[X(\pi_2(\widetilde{\Xi}_t))^p]^{1/p}.\]
	By Lemma~\ref{lem_projections}(a), the right-hand side equals
\[\E[X_t^p \mid \Xi_0 = \delta_A]^{1/p}+ \E[X_t^p \mid \Xi_0 = \delta_B]^{1/p}.\] This completes the proof.
\end{proof}

%\begin{remark} \emph{We stress that the coupling given in the previous proof does not provide a stochastic domination (in any sense) of the herds process starting from $\Xi'$ over the one starting from $\Xi$, and we do not know if such a stochastic domination holds. }
%\end{remark}

\begin{remark}
    {It is worth mentioning a curious fact here, which is the main reason for introducing the two-type herds process. Let~$A,B \in P_\mathsf{f}(\mathbb{T}^d)$ be disjoint.
Then, a natural guess would be that the number of particles in the herds process starting from $\delta_A + \delta_B$ should stochastically dominate the number of particles in the process starting from $\delta_{A \cup B}$, since in the former case there is a priori more space for the particles to spread. However, there seems to be no simple proof of this fact, and it is not even clear whether it should be true or not. In a natural choice of coupling, one would try to map particles of the process started from~$\delta_{A \cup B}$ injectively into particles of the  process started from~$\delta_A + \delta_B$, and make it so that when a particle of the former process dies or gives birth, its image under the injective mapping does the same. However, this does not work. For instance, when starting from $\delta_{A \cup B}$, it could be that at a later time, a single split would separate more particles than in the process starting from $\delta_A + \delta_B$, which in turn could after another birth event give rise to a larger number of particles in the process starting from $\delta_{A \cup B}$. } 
\end{remark}

\begin{remark}\label{rem.ktype}
{The notion of two-type herds process can of course be generalized to a multi-type herds process. Given any integer $k\ge 1$, the $k$-type herds process is the continuous-time Markov chain on 
\[\mathcal{S}_k := \{\xi:P_{\mathsf{f},k}(\T^d)   \to \N_0,\text{ with } \textstyle{\sum_{(A_1,\dots,A_k)}} \xi(A_1,\dots,A_k) < \infty\},\]
where
\[P_{\mathsf{f},k}(\mathbb{T}^d):= \{(A_1,\ldots,A_k): A_1,\ldots,A_k \subset \mathbb{T}^d,\;\cup_{j=1}^k A_j \neq \varnothing\}\]
and where, like in the two-type herds process, every type obeys  birth and death mechanisms as in the original herds process, ignoring the other types, but they all share the same splitting events. Naturally, an analogue of Lemma~\ref{lem_projections} holds as well in this general setting. This remark will be used in the proof of Lemma~\ref{lem_separate_two_again} below.} 
\end{remark}

\section{Analysis of the herds process through a growth index}

\subsection{Definition and properties of growth index}
\label{s_first_properties_phi}
For any~$p \ge 1$, define the \textbf{growth index of order~$p$} as
\begin{equation}\label{def_varphi}
	\varphi_p=\varphi_p(\lambda,\mathsf{v}) := \inf_{t \ge 0}\; \E_{\lambda,\mathsf v}[X_t^p]^{1/t}.
\end{equation}
In case~$p=1$, we omit the subscript, so that
\[\varphi := \varphi_1.\]
Note that we have
\begin{equation}\label{eq_obvious_ineq}\varphi_p^t \le \mathbb{E}[X_t^p],\qquad t \ge 0,\; p \ge 1.\end{equation} 
Moreover,~\eqref{eq_subadd} implies that
\begin{equation}
\label{eq_repeat_subadd}
\E[X_{t+s}^p] \le \E[X_t^p]\cdot \E[X_s^p],\qquad t,s \ge 0,\; p \ge 1,
\end{equation}
and Fekete's lemma then ensures that
\begin{equation}\label{eq_lim}
	\lim_{t \to \infty} \E[X_t^p]^{1/t} = \varphi_p,\qquad p \ge 1.
\end{equation}

It will also be useful to observe that
\begin{equation}\label{eq_jensen}
[1,\infty) \ni p \mapsto \varphi_p^{1/p} \text{ is non-decreasing.}
\end{equation}
To see this, let~$p \ge q \ge 1$. We bound, for any~$t \ge 0$:
\[\E[X_t^p] = \E[(X_t^q)^{p/q}] \ge \E[X_t^q]^{p/q},\]
by Jensen's inequality. Hence,
\[
(\E[X_t^p]^{1/t})^{1/p} \ge (\E[X_t^q]^{1/t})^{1/q}.
\]
By taking~$t \to \infty$ and using~\eqref{eq_lim}, we obtain~$\varphi_p^{1/p} \ge \varphi_q^{1/q}$, as desired.

For the rest of this section, we focus on the growth index with~$p = 1$. We will analyse higher values of~$p$ in Section~\ref{s_higher_moments}.

We state a lemma with an upper bound that, apart from a constant prefactor, matches~\eqref{eq_obvious_ineq} in the case~$p=1$:
\begin{lemma} \label{lem_revineq}
There is a constant~$C = C(\lambda,\mathsf{v})$, such that for any $t\ge 0$, 
\begin{equation}\label{eq_const}\E[X_t] \le C\cdot \varphi^t.\end{equation}
\end{lemma}

Before proving this, we state and prove an auxiliary result, concerning the expected number of herds containing a single particle, after one time unit of the dynamics has elapsed.
\begin{lemma}\label{lem_only_one_sep} There is a constant~$\rho = \rho(\lambda,\mathsf{v})$ such that for any~$\xi \in \mathcal{S}$, we have
    \begin{equation}
        \label{eq_only_one}
        \mathbb{E}\left[\left. \sum_{A:|A|=1} \Xi_1(A)\; \right| \; \Xi_0 = \xi \right] \ge \rho \cdot X(\xi). 
    \end{equation}
\end{lemma}
\begin{proof}
    By Lemma~\ref{lem_decompose} and the linearity of expectation, it suffices to show that there exists~$\rho>0$ such that for any~$A \in P_\mathsf{f}(\mathbb{T}^d)$,
    \begin{equation*}\mathbb{E}\left[\left. \sum_{A':|A'|=1} \Xi_1(A')\; \right| \; \Xi_0 = \delta_A \right] \ge \rho \cdot |A|. 
\end{equation*}
To prove this, fix~$A \in P_\mathsf{f}(\mathbb{T}^d)$, and note that the left-hand side above is larger than
\[
\sum_{u \in A} \mathbb{E}[\Xi_1(\{u\}) \mid \Xi_0 = \delta_A] \ge \sum_{u \in A} \mathbb{P}(\Xi_1(\{u\}) > 0 \mid \Xi_0 = \delta_A).
\] 
It is easy to see that there is a constant~$\rho > 0$, depending only on~$\lambda$ and~$\mathsf{v}$, such that the probability in the sum in the right-hand side is larger than~$\rho$ (for any~$A$ and~$u$). This is achieved by prescribing that the particle present at~$u$ at time~$0$ does not die or give birth until time~$1$, and moreover this particle becomes separated, through successive splits in the edges that are incident to~$u$, of any other particle in its herd. This shows that the right-hand side above is larger than~$\rho|A|$, completing the proof.
\end{proof}

\begin{proof}[Proof of Lemma~\ref{lem_revineq}] 
For any~$s,t \ge 0$ we have
$$\E[X_{t+s}\mid \mathcal F_s]= \sum_A \Xi_s(A) \cdot \mathbb{E}[X_t \mid \Xi_0 = \delta_A] \ge \sum_{A:|A|=1}\Xi_s(A) \cdot \E[X_t],$$
where the equality follows from Lemma~\ref{lem_decompose} and the Markov property, and the inequality from Lemma~\ref{lem_auto}.
Then, by taking expectations and using Lemma~\ref{lem_only_one_sep}, we obtain that
\[\mathbb{E}[X_{t+s}] \ge \mathbb{E} \left[ \sum_{A:|A|=1} \Xi_s(A) \right] \cdot \mathbb{E}[X_t].\]
In case~$s \ge 1$, by Lemma~\ref{lem_only_one_sep} and the Markov property, the right-hand side is larger than
\[\rho \cdot \mathbb{E}[X_{s-1}] \cdot \mathbb{E}[X_t].\]
Further, by Proposition~\ref{prop_subad}, letting~$\kappa:= 1/\mathbb{E}[X_1]$, the above is larger than
\[\kappa \rho \cdot \mathbb{E}[X_s] \cdot \mathbb{E}[X_t].\]

Using this recursively, we obtain, for any~$t > 0$ and~$n \in \mathbb{N}$,
\[\mathbb{E}[X_{nt}] \ge (\kappa \rho)^{n-1}\cdot \mathbb{E}[X_t]^n,\]
so that
\[\mathbb{E}[X_t] \le (\kappa \rho)^{-\frac{n-1}{n}} \cdot  \mathbb{E}[X_{nt}]^\frac{1}{n} = (\kappa \rho)^{-\frac{n-1}{n}}  \cdot (\mathbb{E}[X_{nt}]^\frac{1}{nt})^t.\]
Using~\eqref{eq_lim}, we obtain 
$$\E[X_t] \le  \liminf_{n\to \infty} \left( (\kappa \rho)^{-\frac{n-1}{n}}\cdot (\mathbb{E}[X_{nt}]^\frac{1}{nt})^t 
 \right) = C\cdot  \varphi^t,$$
with $C=\tfrac{1}{\kappa \rho}$, proving the lemma.  
\end{proof}

Together with the fact that~$(\lambda,\mathsf{v}) \mapsto \E_{\lambda,\mathsf v}[X_t]$ is continuous, we deduce that~$(\lambda,\mathsf{v}) \mapsto \varphi(\lambda,\mathsf{v})$ is both upper- and lower-semicontinuous, hence it is continuous. Moreover, we have the following simple characterization of the supercritical regime. Recall the definition of~$\bar{\lambda}(\mathsf{v})$ from Definition~\ref{def_critical_value}.

\begin{lemma}\label{lem_surv_cond}
For any~$\lambda > 0$ and~$\mathsf{v} > 0$, the following are equivalent: 
\begin{itemize}
\item[(a)] the herds process survives with positive probability; 
\item[(b)] $\varphi >1$;
\item[(c)]  $\E[X_t] \xrightarrow{t \to \infty} \infty$.
\end{itemize} 
Consequently, for any~$\mathsf{v} > 0$,
\begin{equation}\label{varphi_lambda_crit}
\varphi(\bar{\lambda}(\mathsf{v}),\mathsf{v}) = 1.
\end{equation}
\end{lemma}
\begin{proof}
The fact that (b) and (c) are equivalent follows from the fact that for any $t\ge 0$, $\varphi^t \le \E[X_t] \le C \varphi^t$.

Now, if (a) holds, then we must also have~$X_t \xrightarrow{t \to \infty} \infty$ with positive probability,  
as otherwise using the conditional Borel-Cantelli Lemma and a standard argument, we would get a contradiction. Then (c) follows from Fatou's Lemma.

Conversely, assume that (c) holds. 
Let $Z_t := \sum_{A:|A|=1} \Xi_t(A)$ denote the number of herds in $\Xi_t$ with a single particle in them. By Lemma~\ref{lem_only_one_sep}, we have~$\E[Z_{t+1}] \ge \rho\cdot \E[X_t]$. It follows that there exists some $T>0$ such that $\E[Z_T]\ge 2$. Since different herds evolve independently of each other, we deduce that $(Z_{nT})_{n\ge 0}$ dominates a supercritical branching process, 
and therefore survives forever with positive probability. Since it is dominated by $(X_{nT})_{n\ge 0}$, we get that (a) is satisfied. 

Having established the equivalence between (a) and (b), the equality~\eqref{varphi_lambda_crit} follows from the continuity of~$\varphi$.
\end{proof}

\subsection{Strict monotonicity of growth index}

Our goal in this section is to prove the following result. 
\begin{proposition}\label{prop_strict_mon}
The map $(\lambda,\mathsf{v}) \mapsto \varphi(\lambda, \mathsf{v})$ is strictly increasing in both arguments. 
\end{proposition}

Before discussing the proof of this, let us see how it allows us to prove Theorem~\ref{theo.2}:
\begin{proof}[Proof of Theorem~\ref{theo.2}]
Let~$\mathsf{v}' > \mathsf{v} > 0$. We have
\[\varphi(\bar{\lambda}(\mathsf{v}'),\mathsf{v}') = 1 = \varphi(\bar{\lambda}(\mathsf{v}),\mathsf{v}) < \varphi(\bar{\lambda}(\mathsf{v}),\mathsf{v}'), \]
where the two equalities are given by~\eqref{varphi_lambda_crit} and the inequality by the strict monotonicity of~$\varphi$. Again using the strict monotonicity of~$\varphi$, we conclude from~$\varphi(\bar{\lambda}(\mathsf{v}'),\mathsf{v}')< \varphi(\bar{\lambda}(\mathsf{v}),\mathsf{v}')$ that~$\bar{\lambda}(\mathsf{v}') < \bar{\lambda}(\mathsf{v}).$
\end{proof}

The proof of Proposition~\ref{prop_strict_mon} will require several steps. We start with a definition. 

\begin{definition}\label{def_g}
	Fix the parameters~$\mathsf{v}$ and~$\lambda$ of the herds process. Given~$\xi \in \mathcal{S}$ and~$t > 0$, we let
	\begin{equation}\label{eq_def_g}
	g_\mathsf{v}(\xi,t):= \sum_{A \in P_\mathsf{f}(\T^d)} \xi(A)\sum_{\substack{e \text{ active}\\\text{edge of }A}} (\mathbb{E}_{\lambda, \mathsf{v}}[X_t\mid \Xi_0 = \delta_{A_{e,1}} + \delta_{A_{e,2}}]- \mathbb{E}_{\lambda,\mathsf{v}}[X_t\mid \Xi_0 = \delta_{A}]),\end{equation}
and
	\begin{equation}\label{eq_def_h}
		h_\lambda(\xi,t):= \sum_{A \in P_\mathsf{f}(\T^d)} \xi(A) \sum_{\substack{x \in A,\; y \notin A,\\ x \sim y}} (\E_{\lambda, \mathsf{v}}[X_t \mid \Xi_0 = \delta_{A \cup \{y\}}] - \E_{\lambda, \mathsf{v}}[X_t \mid \Xi_0 = \delta_A])
	\end{equation}
	(we omit~$\lambda$ from the notation for~$g_\mathsf{v}$ and we omit~$\mathsf{v}$ from the notation for~$h_\lambda$).
\end{definition}
The functions~$g_\mathsf{v}$ and~$h_\lambda$ give a measure of the total impact on the number of particles at time~$t$ of splitting the herds of~$\xi$ at time zero and having a birth at time zero, respectively.

\begin{proposition}
	\label{prop.derivative}
	For any~$T > 0$ we have
	\begin{equation}\label{eq_derivative_v}
		\frac{\partial}{\partial\mathsf{v}}\E_{\lambda,\mathsf{v}}[X_T] = \int_0^T \mathbb E_{\lambda,\mathsf{v}}[ g_\mathsf{v}(\Xi_t,T-t)]\;\mathrm{d}t,
	\end{equation}
	and
	\begin{equation}\label{eq_derivative_l}
		\frac{\partial}{\partial\lambda}\E_{\mathsf{v},\lambda}[X_T] = \int_0^T \mathbb E_{\mathsf{v},\lambda}[h_\lambda(\Xi_t,T-t)]\;\mathrm{d}t.
	\end{equation}
\end{proposition}
The proof is postponed to Section~\ref{s_proof_derivative}, and we now prove another intermediate result. 

\begin{lemma}\label{lem_addherd}
	Let~$\xi'$ consist of a single herd with exactly two particles, which are neighbors, 
	and~$\xi''$ consist of two herds, each containing a single particle. Then, there exists~$\gamma > 0$ (depending continuously on~$\lambda$ and~$\mathsf{v}$) such that for any $t\ge 1$, 
	\[ \E[X_t \mid \Xi_0 = \xi''] \ge  \E[X_t \mid \Xi_0 = \xi'] + \gamma\cdot \E[X_t \mid X_0 = \delta_{\{o\}}].\] 
\end{lemma}
\begin{proof}
	Fix~$v \sim o$. Without loss of generality, we assume that~$\xi'$ consists of the herd~$\{o,v\}$ and~$\xi''$ consists of the herds~$\{o\}$ and~$\{v\}$. {We define a coupling of the two herds processes $(\Xi'_s)$ and $(\Xi''_s)$ starting respectively from $\xi'$ and $\xi''$:}  first we take independent random variables, as follows:
	\begin{itemize}
		\item $\tau_o$ and~$\tau_v$, both~$\sim \mathrm{Exp}(1)$;
		\item for each~$u \sim o$,~$\tau_{o,u} \sim \mathrm{Exp}(\lambda)$;
		\item for each~$u \sim v$,~$\tau_{v,u} \sim \mathrm{Exp}(\lambda)$;
		\item $\tau_{\mathrm{split}} \sim \mathrm{Exp}(\mathsf{v})$.
	\end{itemize}
	Additionally, let~$\tau'$ denote the minimum of all these random variables, and~$\tau := \min(1,\tau')$. Now, the coupling is defined as follows. In all cases, we let~$(\Xi'_s,\Xi_s'') = (\xi',\xi'')$ for~$s \in [0,\tau)$; the definition of~$(\Xi'_\tau,\Xi_\tau'')$ will be split into several cases, but after time~$\tau$, we let~$\Xi_\tau'$ and~$\Xi_{\tau}''$ continue evolving independently as two herds processes with split rate~$\mathsf{v}$. The definition of $(\Xi_\tau',\Xi_\tau'')$ is as follows: if $\tau =1<\tau'$, then 
	$(\Xi_\tau', \Xi_\tau'') =  (\xi',\xi'')$. Otherwise: 
	\begin{itemize}
		\item if $\tau = \tau_o$, then $\Xi_\tau'= \Xi''_\tau = \delta_{\{v\}}$; similarly if $\tau = \tau_v$, then $\Xi_\tau'= \Xi''_\tau = \delta_{\{o\}}$;
		\item if for some~$u \sim o$,~$\tau = \tau_{o,u}$, then $\Xi_\tau'=\delta_{\{o,u,v\}}$ and $\Xi''_\tau =\delta_{\{o,u\}} + \delta_{\{v\}}$; 
		\item if for some~$u \sim v$,~$\tau=\tau_{v,u}$, then $\Xi_\tau'=\delta_{\{o,u,v\}}$ and $\Xi''_\tau =\delta_{\{o\}} + \delta_{\{u,v\}}$;
		\item if $\tau = \tau_{\mathrm{split}}$, then $\Xi_\tau' = \Xi_\tau''= \delta_{\{o\}} + \delta_{\{v\}}$.
	\end{itemize}

 We let~$\widehat{\P}$ denote a probability measure under which this coupling is defined, and~$\widehat{\E}$ be the associated expectation operator. Also let~$(\mathcal{F}_t)_{t \ge 0}$ denote the natural filtration of~$(\Xi_t',\Xi_t'')_{t \ge 0}$.

Fix~$t \ge 1$. Using Lemma~\ref{lem_general_p} and inspecting all cases concerning~$\tau$, it is easy to check that
\begin{equation}
    \label{eq_aoux1}
    \widehat{\E}[X(\Xi_t'') \mid \mathcal{F}_\tau] \ge \widehat{\E}[X(\Xi_t') \mid \mathcal{F}_\tau]. 
\end{equation}
 
	Define the good event~$E:= \{\tau = \tau_{o,v}\} \cup \{\tau = \tau_{v,o}\}$, and note that on this event, $\Xi''_\tau$  contains $\Xi_\tau'$ plus an additional herd with a single particle in it. Hence,
\begin{equation}\label{eq_aoux2}\widehat{\E}[X(\Xi_t'')\mid \mathcal{F}_\tau] \cdot \mathds{1}_E = \left(\widehat{\E}[X(\Xi_t')\mid \mathcal{F}_\tau] + F(t-\tau) \right) \cdot \mathds{1}_E,\end{equation}
where~$F(s) := \E[X_s\mid \Xi_0 = \delta_{\{o\}}]$. 

We now write
\begin{align*}\E[X_t \mid \Xi_0 = \xi''] &= \widehat{\E}[X(\Xi''_t)] = \widehat{\E} [\widehat{\E}[X(\Xi''_t) \mid \mathcal{F}_\tau] \cdot \mathds{1}_E + \widehat{\E}[X(\Xi''_t) \mid \mathcal{F}_\tau] \cdot \mathds{1}_{E^c}]\end{align*}
and using~\eqref{eq_aoux1} and~\eqref{eq_aoux2}, we bound the right-hand side from below by
\begin{align*}  &\widehat{\E}[ (\widehat{\E}[X(\Xi'_t) \mid \mathcal{F}_\tau] + F(t-\tau)) \cdot \mathds{1}_E +  \widehat{\E}[X(\Xi'_t) \mid \mathcal{F}_\tau] \cdot \mathds{1}_{E^c} ]
\\
&\ge \widehat{\E}[X(\Xi'_t)] + \widehat{\E}[F(t-\tau) \cdot \mathds{1}_E]\\
		&\ge \E[X_t \mid \Xi_0 = \xi'] +\widehat{\mathbb{P}}(E)\cdot  \min_{t-1 \le s \le t} F(s)  .
	\end{align*}
 Using~\eqref{eq_subadd} we have that~$F(s) \cdot F(t-s) \ge F(t)$ for any $s \in [t-1,t]$, which gives
 \[\min_{t-1 \le s \le t} F(s) \ge F(t) \cdot \left(\min_{0 \le r \le 1} F(r)\right)^{-1}.\]
 The lemma is thus proved with~$\gamma := \widehat{\P}(E)\cdot  (\min_{0 \le t \le 1}F(r))^{-1} $.
\end{proof}

\begin{proof}[Proof of Proposition~\ref{prop_strict_mon}]
We only prove the strict monotonicity in $\mathsf v$, the argument for the strict monotonicity in $\lambda$ is entirely similar. We start with some basic observations. Let~$X_t'$ denote the number of herds in the herds process at time~$t$ which contain exactly two particles, these particles being neighbors. We claim that
\begin{equation}\label{X'tXt}
	\E[X_t'] \ge c\cdot  \E[X_{t-1}]\quad \text{for any } t \ge 1, 
\end{equation}
with $c$ some positive constant depending continuously on  $\mathsf v$. 
To see this, recall that if we denote by $Z_t$ the number of herds of~$\Xi_t$ that contain a single particle, then (as in the proof of Lemma~\ref{lem_revineq}) one has $\E[Z_{t-1/2}] \ge c_1 \cdot \E[X_{t-1}]$, for some constant $c_1>0$, depending continuously on $\mathsf v$. Since in any half unit of time, a herd with a single particle can be transformed in a herd with two neighboring particles, at a constant price (depending continuously on $\mathsf v$), this gives the claim~\eqref{X'tXt}.

Recalling the definition of $g_{\mathsf v}$ in Definition~\ref{def_g}, by Lemma~\ref{lem_addherd} we have 
$$g_{\mathsf v}(\Xi_s,t-s) \ge \gamma \cdot X'_s \cdot \E[X_{t-s}].$$
Together with Proposition~\ref{prop.derivative} and~\eqref{X'tXt}, this gives for any  
$\mathsf v>0$, and any $t\ge 1$, with $F(\mathsf v,t)=\mathbb E_{\mathsf v}[X_t]$, 
\begin{equation*}
	\frac{\partial F}{\partial \mathsf{v}}(\mathsf{v},t) \ge c\gamma \cdot \int_0^t \E[X_{s-1}] \cdot \E[X_{t-s}] \, ds.
\end{equation*}
Applying then Proposition~\ref{prop_subad}, gives  
\begin{equation*}
	\frac{\partial F}{\partial \mathsf{v}}(\mathsf{v},t) \ge c' \cdot t \cdot F(\mathsf{v},t),
\end{equation*}
where~$c'>0$ depends continuously on~$\mathsf{v}$. 
This implies that, for any fixed~$0<\mathsf{v}_1 < \mathsf{v}_2$, 
\[\log \frac{F(\mathsf{v}_2,t)}{F(\mathsf{v}_1,t)} \ge \rho \cdot t \cdot (\mathsf{v}_2 - \mathsf{v}_1), \]
for some $\rho>0$. 
Thus, by taking the limit as~$t \to \infty$ and using~\eqref{eq_lim}, we get 
\[\varphi(\mathsf{v}_1,\lambda) \le e^{-\rho(\mathsf{v}_2 - \mathsf{v}_1)}\cdot \varphi(\mathsf{v}_2,\lambda).\]
\end{proof}

\subsection{Proof of derivative formulas}\label{s_proof_derivative}
We now turn to establishing~\eqref{eq_derivative_v} and~\eqref{eq_derivative_l}. The proofs are entirely analogous, so we only do the former. A few of the more technical points of the proof are done in the appendix.

Recall the definition of the function~$g_{\mathsf{v}}(\xi,t)$ in~\eqref{eq_def_g}. We now define the closely related function:
\begin{equation}
	\label{eq_alternate_g_other}
	\begin{split}
		g_{\mathsf{v},\varepsilon}(\xi,t):= \sum_{A \in P_\mathsf{f}(\T^d)} \xi(A)\sum_{\substack{e \text{ active}\\\text{edge of }A}} &(\mathbb{E}_{\lambda, \mathsf{v}+\varepsilon}[X_t\mid \Xi_0 = \delta_{A_{e,1}} + \delta_{A_{e,2}}]\\[-.7cm]&\hspace{2cm} - \mathbb{E}_{\lambda,\mathsf{v}}[X_t\mid \Xi_0 = \delta_{A}]),
	\end{split}
\end{equation}
where~$\mathsf{v} > 0$ and~$\varepsilon > 0$. Note that the expression for~$g_{\mathsf{v},\varepsilon}$ only differs from that for~$g_{\mathsf{v}}$ in that in the former, the first expectation that appears in the right-hand side is under parameters~$\lambda$,~$\mathsf{v} + \varepsilon$ rather than~$\lambda$,~$\mathsf{v}$.

The following integrability result will be useful. The proof is done in the appendix.
\begin{lemma}\label{lem_integral_g}
For any~$T > 0$ and~$k \ge 1$, we have
	\[\E_{\lambda,\mathsf{v}}\left[\max_{0 \le t \le T}|g_\mathsf{v}(\Xi_t,T-t)|^k\right] < \infty \quad \text{and}\quad \E_{\lambda,\mathsf{v}}\left[\max_{0 \le t \le T} |g_{\mathsf{v},\varepsilon}(\Xi_t,T-t)|^k\right] < \infty.\] 
\end{lemma}

The following definition will give an alternative expression for~$g_\mathsf{v}$ and~$g_{\mathsf{v},\varepsilon}$, in~\eqref{eq_alternate_g} and~\eqref{eq_alternate_g_eps} below.

\begin{definition}
	\label{def_n_for_g}
	Given~$\xi \in \mathcal{S}$, let~$\mathcal{G}(\xi)$ denote the set of all herd configurations~$\xi'$ that can be obtained by performing a single split on~$\xi$. Given~$\xi' \in \mathcal{G}(\xi)$, let~$A$ be the (unique) herd shape that is split into two to obtain~$\xi'$ from~$\xi$; let~$\mathsf{m}(\xi,\xi'):= \xi(A)$.
\end{definition}
To further clarify this definition, fix an enumeration~$\xi = \sum_{i=1}^m \delta_{A_i}$ of~$\xi$. Then,~$\mathcal{G}(\xi)$ is the set of~$\xi' \in \mathcal{S}$ with
	\[\xi' = \sum_{i \neq j} \delta_{A_i} + \delta_{(A_j)_{e,1}} + \delta_{(A_j)_{e,2}},\]
	where~$j \in \{1,\ldots, m\}$ and~$e$ is an active edge of~$A_j$. For~$\xi'$ as in the above display, we have~$\mathsf{m}(\xi,\xi')=\xi(A_j)$.

	We now observe that, using Lemma~\ref{lem_decompose}, we can rewrite
\begin{equation}
	\label{eq_alternate_g}
	g_\mathsf{v}(\xi,t) = \sum_{\xi' \in \mathcal{G}(\xi)} \mathsf{m}(\xi,\xi') \cdot (\E_{\lambda,\mathsf{v}}[X_t\mid \Xi_0 = \xi'] - \E_{\lambda,\mathsf{v}}[X_t \mid \Xi_0 = \xi])
\end{equation}
and
\begin{equation}
	\label{eq_alternate_g_eps}
	g_{\mathsf{v},\varepsilon}(\xi,t) = \sum_{\xi' \in \mathcal{G}(\xi)} \mathsf{m}(\xi,\xi') \cdot (\E_{\lambda,\mathsf{v}+\varepsilon}[X_t\mid \Xi_0 = \xi'] - \E_{\lambda,\mathsf{v}}[X_t \mid \Xi_0 = \xi]).
\end{equation}

Fix~$\mathsf{v} > 0$,~$\lambda > 0$ and~$\varepsilon > 0$. We will now construct a coupling~$(\mathcal{V}_t,\mathcal{W}_t)_{t \ge 0}$ under a probability measure~$\widehat{\P}$ (with dependence on~$\mathsf{v},\lambda,\varepsilon$ omitted) so that~$(\mathcal{V}_t)$ is a herds process with parameters~$\lambda$,~$\mathsf{v}$, and~$(\mathcal{W}_t)$ is a herds process with parameters~$\lambda$,~$\mathsf{v}+\varepsilon$; both these processes are started from a single herd with a single particle (at the root of~$\T^d$). 

\begin{definition}[Coupling~$(\mathcal{V}_t,\mathcal{W}_t)$]
	Take a probability space with probability measure~$\widehat{\P}$ under which a herds process~$(\mathcal{V}_t)_{t \ge 0}$ with parameters~$\lambda$,~$\mathsf{v}$ is defined, started from~$\delta_{\{o\}}$. Assume that the split jumps of~$(\mathcal{V}_t)$ are given as follows: splitting instructions arise with rate~$\mathsf{v}+\varepsilon$ (rather than~$\mathsf{v}$), but they are rejected with probability~$\frac{\varepsilon}{\mathsf{v}+\varepsilon}$. Let
	\[\tau_\mathrm{sep} := \inf\{t \ge 0:\; \text{a splitting instruction is rejected at time $t$}\}.\]
	We define~$(\mathcal{W}_t)_{t \ge 0}$ as follows. For~$0 \le t < \tau_\mathrm{sep}$, we set~$\mathcal{W}_t = \mathcal{V}_t$. At time~$\tau_\mathrm{sep}$, this process obeys the splitting instruction that was rejected by~$(\mathcal{V}_t)$. We then let~$(\mathcal{W}_t)_{t \ge \tau_\mathrm{sep}}$ continue evolving from~$\mathcal{W}_{\tau_\mathrm{sep}}$ as a herds process with parameters~$\lambda$,~$\mathsf{v}+\varepsilon$, independent of~$(\mathcal{V}_t)_{t \ge \tau_\mathrm{sep}}$. Finally, we let 
	\[\mathcal X_t = X(\mathcal V_t),\qquad \mathcal Y_t = X(\mathcal W_t),\qquad t \ge 0.\] 
\end{definition}

For the rest of this section, we fix~$T > 0$.

\begin{definition}
	\label{def_of_A}
We define the process~$(\mathcal{A}_t)_{0 \le t \le T}$ as
\begin{equation*}
	\mathcal{A}_t:= \mathds{1}\{\tau_\mathrm{sep} \le t\}\cdot \widehat{\E}[\mathcal{Y}_T - \mathcal{X}_T \mid \mathcal{F}_{\tau_\mathrm{sep}}],\quad 0 \le t \le T.
\end{equation*}
\end{definition}

That is, in the event~$\{\tau_{\mathrm{sep}} > T\}$ we have~$\mathcal{A}_t \equiv 0$, whereas in~$\{\tau_\mathrm{sep} \le T\}$, this process takes just two values:~$\mathcal{A}_t = 0$ for~$t \in  [0,\tau_\mathrm{sep})$ and~$\mathcal{A}_t = \widehat{\E}[\mathcal{Y}_T - \mathcal{X}_T \mid \mathcal{F}_{\tau_\mathrm{sep}}]$ for~$t \in [\tau_\mathrm{sep},T]$. Our interest in this process stems from the fact that
	\begin{equation}\label{eq_stems}\begin{split}\E_{\lambda,\mathsf{v}+\varepsilon}[X_T] - \E_{\lambda,\mathsf{v}}[X_T] &= \widehat{\E}[\mathcal{Y}_T - \mathcal{X}_T] \\
		&= \widehat{\E}[(\mathcal{Y}_T - \mathcal{X}_T)\cdot \mathds{1}\{\tau_\mathrm{sep} \le T\}]\\
	&= \widehat{\E}[\widehat{\E}[(\mathcal{Y}_T - \mathcal{X}_T)\mid \mathcal{F}_{\tau_\mathrm{sep}}] \cdot \mathds{1}\{\tau_\mathrm{sep} \le T\}] = \widehat{\E}[\mathcal{A}_T].
	\end{split}
	\end{equation}
	We now compute the right derivative with respect to time of the conditional expectation of~$\mathcal{A}_t$. This lemma is where the function~$g_{\mathsf{v},\varepsilon}$ enters the picture.

\begin{lemma}\label{lem_cond_derivative_A}
	\label{lem_right_derivative_A}
	For any~$t \in [0,T)$, on the event~$\{\tau_\mathrm{sep} > t\}$ we have
	\begin{equation}
		\label{eq_cond_derA}
		\left.\frac{\mathrm{d}}{\mathrm{d}s} \widehat{\E}[\mathcal{A}_{t + s}\mid \mathcal{F}_t ] \right|_{s = 0+} = \varepsilon\cdot g_{\mathsf{v},\varepsilon}(\mathcal{V}_t, T-t).
	\end{equation}
\end{lemma}

\begin{proof}

	We abbreviate~\eqref{eq_alternate_g_eps} by writing
	\[g_{\mathsf{v},\varepsilon}(\xi,s) = \sum_{\xi' \in \mathcal{G}(\xi)} \mathsf{m}(\xi,\xi')\cdot \beta(\xi,\xi',s),\]
	where
	\[\beta(\xi, \xi', s):= \E_{\lambda,\mathsf{v}+\varepsilon}[X_s \mid \Xi_0 = \xi'] - \E_{\lambda,\mathsf{v}}[X_s \mid \Xi_0 = \xi],\quad  \xi \in \mathcal{S},\;\xi' \in \mathcal{G}(\xi),\;s > 0.\]

	Fix~$s > 0$ so that~$t+s \le T$, and fix~$\xi \in \mathcal{S}$. For each~$\xi' \in \mathcal{G}(\xi)$, let~$E(\xi')$ be the event that:
	\begin{itemize}
		\item $\tau_\mathrm{sep} \in (t,t+s]$,
		\item the first jump of~$(\mathcal{V}_r,\mathcal{W}_r)_{t \le r \le t+s}$ is the one that occurs at time~$\tau_\mathrm{sep}$, and
		\item $\mathcal{W}_{\tau_\mathrm{sep}} = \xi'$.
	\end{itemize}

	On the event~$\{\tau_\mathrm{sep} > t,\; \Xi_t = \xi\}$, we have
	\begin{align*}
		\widehat{\E}[\mathcal{A}_{t + s}\mid \mathcal{F}_t] - \widehat{\E}[\mathcal{A}_{t}\mid \mathcal{F}_t]&=\widehat{\E}[\mathcal{A}_{t + s}\mid \mathcal{F}_t] \\
		&=\sum_{\xi' \in \mathcal{G}(\xi)} \widehat{\E}[\mathds{1}_{E(\xi')} \cdot \beta(\xi,\xi',T- \tau_\mathrm{sep})\mid \mathcal{F}_t] + o(s),
	\end{align*}
	where the~$o(s)$ term (which of course refers to when~$s \to 0$) comes from events where there are multiple jumps in~$(t,t+s]$.
	The above sum equals
	\begin{equation}
		\label{eq_sigma_first} \begin{split}&\sum_{\xi' \in \mathcal{G}(\xi)} \beta(\xi,\xi',T-t)\cdot \widehat{\P}(E(\xi')\mid \mathcal{F}_t) \\
			&+ \sum_{\xi' \in \mathcal{G}(\xi)} \widehat{\E}[\mathds{1}_{E(\xi')}\cdot (\beta(\xi,\xi',T-\tau_\mathrm{sep}) - \beta(\xi,\xi',T-t)) \mid \mathcal{F}_t]. \end{split}
	\end{equation}
	We will treat the two sums separately. The absolute value of the second sum  is bounded by
	\[\left(\sup_{\xi' \in \mathcal{G}(\xi),\; u \in (t,t+s]} |\beta(\xi,\xi',T-u) - \beta(\xi,\xi',T-t)|\right)\cdot \sum_{\xi' \in \mathcal{G}(\xi)} \widehat{\P}(E(\xi')\mid \mathcal{F}_t).\]
	As~$s \to 0$, the supremum tends to zero and the sum is bounded by 1, so the whole expression is~$o(s)$.

	We now turn to the first sum in~\eqref{eq_sigma_first}. As~$s\to 0$ we have
	\[\widehat{\P}(E(\xi') \mid \mathcal{F}_t) = (\mathsf{v}+\varepsilon)\cdot \frac{\varepsilon}{\mathsf{v}+\varepsilon}\cdot \mathsf{m}(\xi,\xi')\cdot s + o(s) = \varepsilon s \cdot \mathsf{m}(\xi,\xi') + o(s),\]
	so the sum equals
	\[\varepsilon s \sum_{\xi' \in \mathcal{G}(\xi)} \mathsf{m}(\xi,\xi')\cdot \beta(\xi,\xi',T-t) + o(s) = \varepsilon s \cdot g_{\mathsf{v},\varepsilon}(\xi, T-t) + o(s).\]
	We have thus proved that
	\[ \mathds{1}\{\tau_\mathrm{sep} > t\} \cdot  (\widehat{\E}[\mathcal{A}_{t+s} \mid \mathcal{F}_t] - \widehat{\E}[\mathcal{A}_t \mid \mathcal{F}_t]) =\mathds{1}\{\tau_\mathrm{sep} > t\}\cdot (\varepsilon s\cdot g_{\mathsf{v},\varepsilon}(\mathcal{V}_t,T-t)+o(s)),\]
	so~\eqref{eq_cond_derA} follows. 
\end{proof}

Next, we obtain the expression for the derivative with respect to time of the (non-conditional) expectation of~$\mathcal{A}_t$.
\begin{lemma}\label{lem_derivative_A} For~$t \in [0,T)$, we have
	\begin{equation}
		\label{eq_derivative_A}
		\frac{\mathrm{d}}{\mathrm{d}t} \widehat{\E}[\mathcal{A}_t] = \varepsilon \cdot \widehat{\E}[\mathds{1}\{\tau_{\mathrm{sep}} > t\}\cdot g_{\mathsf{v},\varepsilon}(\mathcal{V}_t, T-t)].
	\end{equation}
\end{lemma}
The first step in establishing this lemma is noting that, for~$0 \le t < t+s \le T$, 
\[
	 \frac{\widehat{\E}[\mathcal{A}_{t+s}] - \widehat{\E}[\mathcal{A}_t]}{s} =  \widehat{\E}\left[\mathds{1}\{\tau_{\mathrm{sep}} > t\} \cdot \frac{\widehat{\E}[\mathcal{A}_{t+s}\mid \mathcal{F}_t]}{s} \right],
\]
since~$\mathcal{A}_{t+s} = \mathcal{A}_t$ on~$\{\tau_\mathrm{sep} \le t\}$, and~$\mathcal{A}_t = 0$ on~$\{\tau_\mathrm{sep} > t\}$.
We would now like to take~$s$ to zero (from the right only, at least at first) and use Lemma~\ref{lem_cond_derivative_A}, but we need to exchange the limit and the expectation; formally:
\begin{equation}
	\label{eq_exchange}
\lim_{s \to 0+}\widehat{\E}\left[\mathds{1}\{\tau_{\mathrm{sep}} > t\} \cdot \frac{\widehat{\E}[\mathcal{A}_{t+s}\mid \mathcal{F}_t]}{s} \right] = \widehat{\E}\left[\mathds{1}\{\tau_{\mathrm{sep}} > t\} \cdot \lim_{s \to 0+} \frac{\widehat{\E}[\mathcal{A}_{t+s}\mid \mathcal{F}_t]}{s}\right].
\end{equation}
The justification of this exchange is done with a standard dominated convergence argument, but an additional bound is required, so we postpone the full proof of Lemma~\ref{lem_derivative_A} to the Appendix.

\begin{proof}[Proof of~\eqref{eq_derivative_v}]
	Fix~$\varepsilon > 0$. Using~\eqref{eq_stems} and Lemma~\ref{lem_derivative_A}, we have
	\[\frac{\E_{\lambda,\mathsf{v}+\varepsilon}[X_T] - \E_{\lambda,\mathsf{v}}[X_T]}{\varepsilon} = \int_0^T \widehat{\E}[\mathds{1}\{\tau_{\mathrm{sep}} > t\}\cdot g_{\mathsf{v},\varepsilon}(\mathcal{V}_t, T-t)]\;\mathrm{d}t.\]
	By Fubini's theorem (which we can use since we have the integrability condition given in Lemma~\ref{lem_integral_g}), the right-hand side above equals
	\[\widehat{\E}\left[\int_0^{\tau_\mathrm{sep} \wedge T} g_{\mathsf{v},\varepsilon}(\mathcal{V}_t,T-t)\;\mathrm{d}t\right].\]
	We write this as
	\begin{equation}\label{eq_three_exps}\begin{split}
		&\widehat{\mathbb{E}}\left[\int_0^T g_\mathsf{v}(\mathcal{V}_t,T-t)\;\mathrm{d}t\right] \\&+ \widehat{\mathbb{E}}\left[\int_0^T (g_{\mathsf{v},\varepsilon}(\mathcal{V}_t,T-t)-g_\mathsf{v}(\mathcal{V}_t,T-t))\;\mathrm{d}t\right] -\widehat{\mathbb{E}}\left[\int_{\tau_\mathrm{sep}}^T g_{\mathsf{v},\varepsilon}(\mathcal{V}_t,T-t)\;\mathrm{d}t\right].
	\end{split}\end{equation}
	Note that, since the law of~$(\mathcal{V}_t)$ under~$\widehat{\mathbb{P}}$ equals the law of~$(\Xi_t)$ under~$\mathbb{P}_{\lambda,\mathsf{v}}$, we have
	\[\widehat{\mathbb{E}}\left[\int_0^T g_\mathsf{v}(\mathcal{V}_t,T-t)\;\mathrm{d}t\right] = \mathbb{E}_{\lambda,\mathsf{v}}\left[\int_0^T g_\mathsf{v}(\Xi_t,T-t)\;\mathrm{d}t\right].\]
	Hence, the proof will be completed once we prove that the second and third expectations in~\eqref{eq_three_exps} tend to zero as~$\varepsilon \to 0$.

	It is straightforward to show that, for any~$\xi\in \mathcal{S}$ and any~$t \in [0,T]$,
	\[g_{\mathsf{v},\varepsilon}(\xi,t) \xrightarrow{\varepsilon \to 0} g_{\mathsf{v}}(\xi,t).\]
	Combining this with the dominated convergence theorem (using Lemma~\ref{lem_integral_g}), we obtain
	\[\widehat{\mathbb{E}}\left[\int_0^T (g_{\mathsf{v},\varepsilon}(\mathcal{V}_t,T-t)-g_\mathsf{v}(\mathcal{V}_t,T-t))\;\mathrm{d}t\right]\xrightarrow{\varepsilon \to 0} 0.\]
Next, using the Cauchy-Schwarz inequality, we bound
	\begin{align*}&\left|\widehat{\mathbb{E}}\left[\int_{\tau_\mathrm{sep}}^T g_{\mathsf{v},\varepsilon}(\mathcal{V}_t,T-t)\;\mathrm{d}t\right]\right| \\
		&\le T\cdot \widehat{\E}\left[\left(\max_{0\le t \le T} |g_{\mathsf{v},\varepsilon}(\mathcal{V}_t,T-t)|^2 \right) \right]^{1/2}\cdot \widehat{\P}(\tau_\mathrm{sep} \le T)^{1/2}. 
	\end{align*}
	The expectation on the right-hand side is finite by Lemma~\ref{lem_integral_g}, and it is straightforward to check that~$\widehat{\P}(\tau_\mathrm{sep} \le T) \xrightarrow{\varepsilon \to 0} 0$. This completes the proof.
\end{proof}

\subsection{Analysis of higher moments}
\label{s_higher_moments}

Throughout this section, we fix~$\lambda$ and~$\mathsf{v}$ with~$\lambda < \bar{\lambda}(\mathsf{v})$ (recalling the definition of $\bar{\lambda}(\mathsf v)$ in Definition~\ref{def_critical_value}).

We now analyse the growth index~$\varphi_p$ for~$p$ possibly larger than 1.
Our main goal is to prove the following.
\begin{proposition}\label{prop.pmom}
If $\lambda<\bar{\lambda}(\mathsf v)$, then for every~$p \ge 1$ we have
\begin{equation}\label{eq_phi_less_one}
    \varphi_p<1
\end{equation} and
\begin{equation} \label{eq_finite_any_time}
    \sup_{t \ge 0} \E[X_t^p] < \infty.
\end{equation}
\end{proposition}
Note that the case $p=1$ has already been proved: the fact that~$\varphi < 1$ when~$\lambda < \bar{\lambda}(\mathsf{v})$ is given by~\eqref{varphi_lambda_crit} and Proposition~\ref{prop_strict_mon}, and then~\eqref{eq_finite_any_time} follows from~$\varphi < 1$ and~\eqref{eq_const}. 

In order to prove Proposition~\ref{prop.pmom}, we shall need an intermediate result, which we now state.

\begin{lemma} \label{lem_more_refin}
Let $p \ge 2$. There exists~$\mathfrak{C}_p>0$ (depending on~$\lambda$,~$\mathsf{v}$ and~$p$) such that for any $s\ge 12p\log(2d)/\mathsf{v}$, 
$$ \E\left[\sum_{A \in P_\mathsf{f}(\mathbb{T}^d)} \Xi_s(A) \cdot |A|^p\right] \le \mathfrak{C}_p\cdot \big(\E[X_s^{3/2}] + \varphi^s\big). $$ 
\end{lemma}

The proof of this lemma is quite involved, and we postpone it to the next section. We now show how to obtain Proposition~\ref{prop.pmom} from it.

\begin{proof}[Proof of Proposition~\ref{prop.pmom}]
	Fix~$\lambda < \bar{\lambda}(\mathsf{v})$. Note that~\eqref{eq_finite_any_time} follows readily from~\eqref{eq_phi_less_one} and~\eqref{eq_lim}, so we only need to prove~\eqref{eq_phi_less_one}.
 
 By~\eqref{eq_jensen}, it suffices to prove that~$\varphi_p < 1$ for all~$p \in \mathbb{N}$, and the case~$p=1$ is already done. We proceed by induction, fixing~$p \in \{2,3,\ldots\}$ and assuming that~$\varphi_q < 1$ for all~$q \in \{1,\ldots,p-1\}$.

We first claim that there exist $C,c > 0$ such that, for all~$t \ge 0$ and~$\xi \in \mathcal{S}$,
	\begin{equation}\label{eq_refinement}
		\mathbb{E}[X_t^p \mid \Xi_0 = \xi] \le \left( \sum_{A \in P_\mathsf{f}(\mathbb{T}^d)} \xi(A) \cdot |A|^p \right) \cdot \mathbb{E}[X_t^p] + X(\xi)^p \cdot  C\mathrm{e}^{-ct} .
	\end{equation}
 This bound is a refinement of Proposition~\ref{prop_subad}. While in the proof of that proposition we used Minkowski's inequality, here we expand the~$p$-th power of a sum in full and bound the various terms that appear.

 To prove~\eqref{eq_refinement}, let~$t \ge 0$ and~$\xi \in \mathcal{S}$ with enumeration~$\xi = \sum_{i=1}^m \delta_{A_i}$.  By Lemma~\ref{lem_decompose}, $(\Xi_t)_{t \ge 0}$ started from~$\Xi_0 = \xi$ has the same distribution as~$(\Xi^{(1)}_t + \cdots + \Xi^{(m)}_t)_{t \ge 0}$, where $(\Xi^{(1)}_t)_{t \ge 0}$,~$\ldots$,~$(\Xi^{(m)}_t)_{t \ge 0}$ are independent herds processes, with~$\Xi^{(i)}_0 = \delta_{A_i}$ for each~$i$. In particular,
	\begin{align*}
		\mathbb{E}[X_t^p \mid \Xi_0 = \xi] &= \mathbb{E}\left[\left( \sum_{i=1}^m X(\Xi^{(i)}_t) \right)^p\right]\\
		&= \sum_{\substack{(a_1,\ldots,a_m):\\ a_1 + \cdots +a_m = p}} \frac{p!}{a_1! \cdots a_m!} \cdot \prod_{\substack{i \in \{1,\ldots,m\}:\\ a_i > 0}} \mathbb{E}[X(\Xi_t^{(i)})^{a_i}].
	\end{align*}
	By the fact that~$\Xi^{(i)}_0 = \delta_{A_i}$ and Proposition~\ref{prop_subad}, the right-hand side is smaller than
	\begin{equation}
		\sum_{\substack{(a_1,\ldots,a_m):\\ a_1 + \cdots +a_m = p}} \frac{p!}{a_1! \cdots a_m!} \cdot \prod_{\substack{i \in \{1,\ldots,m\}:\\ a_i > 0}} |A_i|^{a_i}\cdot \mathbb{E}[X_t^{a_i}].
\end{equation}
We break this sum as
	\begin{equation} \label{eq_break_sum}
		\E[X_t^p] \cdot \sum_{i=1}^m |A_i|^{p} + \sum_{\substack{(a_1,\ldots,a_m):\\ a_1+ \cdots + a_m = p,\\ a_1,\ldots,a_m < p}} \frac{p!}{a_1! \cdots a_m!} \cdot\prod_{\substack{i\in \{1,\ldots,m\}:\\ a_i > 0}} |A_i|^{a_i} \cdot \mathbb{E}[X_t^{a_i}].
	\end{equation}
	Note that the first sum can be written as
	\[\E[X_t^p] \cdot \sum_{A \in P_\mathsf{f}(\mathbb{T}^d)} \xi(A)\cdot |A|^p.\]
	 Next, the induction hypothesis that~$\varphi_{a} < 1$ for all~$a \in \{1,\ldots, p-1\}$ together with~\eqref{eq_lim}  imply that there exist~$C,c > 0$ (not depending on~$t$) such that~$\E[X_t^a] \le Ce^{-ct}$ for any~$a \in \{1,\ldots, p-1\}$. Hence, the second sum in~\eqref{eq_break_sum} is smaller than
	\begin{align*}
	 \sum_{\substack{a_1,\ldots,a_m \in \N_0:\\a_1+ \cdots + a_m = p,\\ a_1,\ldots, a_m < p}} \frac{p!}{a_1!\cdots a_m!}\cdot \prod_{\substack{i \in \{1,\ldots,m\}:\\ a_i > 0}} |A_i|^{a_i}\cdot C\mathrm{e}^{-ct}.
	\end{align*}
By forgetting the last condition in the summation and increasing the value of~$C$ if necessary, this is smaller than
	\begin{align*}
		& C \mathrm{e}^{-ct} \cdot  \sum_{\substack{a_1,\ldots,a_m \in \N_0:\\a_1+ \cdots + a_m = p}} \frac{p!}{a_1!\cdots a_m!}\cdot \prod_{\substack{i \in \{1,\ldots,m\}:\\ a_i > 0}} |A_i|^{a_i} =  C \mathrm{e}^{-ct} \cdot  X(\xi)^p. 
	\end{align*}
	We have thus proved~\eqref{eq_refinement}. 
 
 Now,~\eqref{eq_refinement} and the Markov property imply that, for any~$s, t \ge 0$, 
  	\begin{equation}\label{eq_refinement_markov} \E[X_{t+s}^p] \le  \E\Big[\sum_A \Xi_s(A)\cdot |A|^p\Big] \cdot \mathbb{E}[X_t^p] +   \E[X_s^p] \cdot C\mathrm{e}^{-ct}.
\end{equation}
Using the bound of Lemma~\ref{lem_more_refin}, increasing the constant~$C$ if necessary, for any~$s,t$ large enough we then have
	\begin{equation}\label{eq_put_together}
		\mathbb{E}[X_{t+s}^p] \le C \left\{ 
		\E[X_s^{3/2}] \cdot \E[X_t^p] + \varphi^s \cdot \E[X_t^p] + \E[X_s^p] \cdot \mathrm{e}^{-ct}
		\right\}.
	\end{equation}
	We also bound
\begin{align*}
    	\E[X_s^{3/2}] & = \E[X_s^{3/2} \cdot \mathds{1}\{X_s > 0\}] \\ &\le \E[(X_s^{3/2})^{4/3}]^{3/4} \cdot \P(X_s > 0)^{1/4} \\
     &\le \E[X_s^2]^{3/4} \cdot \P(X_s \neq 0)^{1/4} \le \E[X_s^2]^{3/4} \cdot \E[X_s]^{1/4} \le C \E[X_s^p]^{3/4} \cdot \varphi^{s/4},
\end{align*}
	where the first inequality is H\"older's, the second inequality follows from~$\mathds{1}\{X_s \neq 0\} \leq X_s$ and the third inequality follows from~$X_s^2 \le X_s^p$ and~\eqref{eq_const}.  We use this bound in~\eqref{eq_put_together}, together with the fact that~$\varphi < 1$, to obtain
	\begin{equation}\label{eq_put_together2}
		\mathbb{E}[X_{t+s}^p] \le C \left\{ 
		\E[X_s^p]^{3/4} \cdot \E[X_t^p] \cdot \mathrm{e}^{-cs} + \E[X_t^p]\cdot \mathrm{e}^{-cs} + \E[X_s^p] \cdot \mathrm{e}^{-ct}
		\right\}
	\end{equation}
	for suitable choices of~$c,C$.

	Now, assume for a contradiction that~$\varphi_p \ge 1$. In that case, by~\eqref{eq_obvious_ineq} we have~$\E[X_r^p] \ge 1$ for all~$r \ge 0$, and from~\eqref{eq_put_together2} we obtain, for large enough~$s,t$ with~$s \le t$:
	\[
		\mathbb{E}[X_{t+s}^p] \le C  
				\E[X_s^p] \cdot \E[X_t^p] \cdot \mathrm{e}^{-cs}. 
					\]
					Using this recursively, we have that for all sufficiently large~$s$ and all~$n \in \N$,
\[ \mathbb{E}[X_{ns}^p] \le C^n \cdot  \E[X_s^p]^n \cdot \mathrm{e}^{-cns}.\]
	Taking both sides to the power~$\frac{1}{ns}$ and letting~$n \to \infty$ (using~\eqref{eq_lim}) gives
	\[
		\varphi_p \le C^{1/s} \cdot \E[X_s^p]^{1/s} \cdot \mathrm{e}^{-c}.
	\]
	Now letting~$s \to \infty$ and again using~\eqref{eq_lim} yields~$\varphi_p \le \varphi_p\cdot  \mathrm{e}^{-c}$, a contradiction.
\end{proof}

Before we turn to the proof of Lemma~\ref{lem_more_refin}, we want to give a consequence of Proposition~\ref{prop_subad}, namely Corollary~\ref{lem.sub.totsize} below, which will be useful in dealing with the contact process on dynamic graphs. For $t\ge 0$, we denote by $N_t$ the number of birth events in the herds process up to time $t$, as in Lemma~\ref{lem_explosion}. Also let~$N_\infty := \lim_{n \to \infty} N_t$.

\begin{lemma}\label{lem_for_p_infinity}
    Let~$p \ge 1$. For any~$t \in [0,\infty]$ and~$\xi \in \mathcal{S}$, we have
    \[\E[N_t^p \mid \Xi_0  =\xi] \le X(\xi)^p \cdot \E[N_t^p]. \]
    Consequently, for any~$s \in [0,\infty)$ and~$t \in [0,\infty]$,
    \begin{equation}\label{eq_s_and_t}
        \E[(N_{s+t}-N_s)^p] \le \E[X_s^p] \cdot \E[N_t^p].
    \end{equation}
\end{lemma}
\begin{proof}
First, recall Remark~\ref{rem.ktype}, which in particular shows that one can dominate the number of particles in a herds process starting from $\xi$ by the total number of particles in a multi-type herds process with~$\ell:=X(\xi)$ types, starting from the configuration $\xi'$ where each particle in $\xi$ represents a distinct type. For this auxiliary process, let~$N_t^{(i)}$ denote the number of births of particles of type~$i$ by time~$t$, for~$i \in \{1,\ldots,\ell\}$. Then,
\[\E[N_t^p \mid \Xi_0 = \xi] \le \E\left[\left( \sum_{i=1}^\ell N_t^{(i)} \right)^p \right] \le \left(\sum_{i=1}^\ell \E[(N_t^{(i)})^p]^{1/p} \right)^p,\]
where we have used Minkowski's inequality. Now, since each type evolves a usual herds process, we have~$\E[(N_t^{(i)})^p] = \E[N_t^p]$ for every~$i$, so the right-hand side above equals~$\ell^p \cdot \E[N_t^p]$, as required.
\end{proof}

\begin{corollary}\label{lem.sub.totsize}
If~$\lambda < \bar{\lambda}(\mathsf{v})$, then~$\E[N_\infty^p]< \infty$ for all~$p \ge 1$.
\end{corollary}
\begin{proof}
Fix~$p \ge 1$.
For any~$x > 1$, let~$t_x := \frac{p+1}{|\log \varphi|} \cdot \log x$ (note that~$\varphi < 1$ since~$\lambda < \bar{\lambda}(\mathsf{v})$, by Lemma~\ref{lem_surv_cond}). Letting~$\tau$ denote the extinction time of the herds process, we first bound
\begin{align*} \P(N_\infty > x) &\le \P(\tau > t_x) + 
\P(N_{t_x} > x).
\end{align*}
The first term on the right-hand side is easy to bound:
\[\P(\tau > t_x) \le \E[X_{t_x}] \stackrel{\eqref{eq_const}}{\le} C\varphi^{t_x} = Cx^{-(p+1)}.\]
Next, we bound
\begin{align*}
    \P(N_{t_x} > x) &\le  \sum_{i=0}^{\lfloor t_x\rfloor -1} \P\left(N_{k+1}-N_k > \frac{x}{t_x}\right) \\[.2cm]
    &\le \sum_{i=0}^{\lfloor t_x \rfloor -1} \frac{\E[(N_{k+1}-N_k)^{p+1}]}{(x/t_x)^{p+1}} \le t_x^{p+2} \cdot x^{-(p+1)} \cdot \sup_{t \ge 0} \E[(N_{t+1}-N_t)^{p+1}].
\end{align*}
Now, the supremum on the right-hand side is finite by~\eqref{eq_finite_any_time} and~\eqref{eq_s_and_t}. We have thus proved that~$\P(N_\infty > x) \le C (\log x)^{p+2} x^{-(p+1)}$ for some~$C > 0$, which implies that~$N_\infty$ has finite~$p$-th moment.
\end{proof}

\subsection{Bound on moments of herd sizes: proof of Lemma~\ref{lem_more_refin}}
We will need several preliminary results, starting with the following.
\begin{lemma} \label{lem_separate_two}
For any~$\alpha \ge 1$ and~$p \ge 1$, there exists $C_{\alpha,p} > 0$ (depending on~$\lambda$,~$\mathsf{v}$,~$\alpha$ and~$p$) such that the following holds. Fix~$u \in \mathbb{T}^d \backslash \{o\}$ and let~$(\widetilde{\Xi}_t)_{t \ge 0}$ be a two-type herds process with rates~$\lambda$ and~$\mathsf{v}$, started from a single herd containing a type-1 particle at the root and a type-2 particle at~$u$, that is,~$\widetilde{\Xi}_0= \delta_{\{o\},\{u\}}$. Then, letting~$t_\alpha:= 12\log(\alpha)/\mathsf{v}$, we have
\begin{equation}\label{eq_bound_with_AB}
\mathbb{E}\left[\left( \sum_{(A,B):A \neq \varnothing} \widetilde{\Xi}_{t_\alpha}(A,B)\cdot |B| \right)^p\;\right] < C_{\alpha,p}\cdot \alpha^{-\mathrm{dist}(o,u)}.
\end{equation}
\end{lemma}

\begin{proof}
	Fix~$\lambda$,~$\mathsf{v}$ and~$\alpha$, let~$t_\alpha := 12\log(\alpha)/\mathsf{v}$ and fix~$u$ as in the statement.
For~$t \ge 0$, let~$Y_t$ denote the number of herds in~$\widetilde{\Xi}_t$ that contain both types, that is,
\[Y_t:= \sum_{\substack{(A,B):\\ A \ne \varnothing,\\ B \ne \varnothing}} \widetilde{\Xi}_t(A,B).\]
It will be useful to note that
	\begin{equation}
		\label{eq_monotone_Y}
		s \le t \quad \Longrightarrow \quad \{Y_s = 0 \} \subseteq \{Y_t = 0\}.
	\end{equation}
Noting that
\[\left( \sum_{(A,B):A \neq \varnothing} \widetilde{\Xi}_{t_\alpha}(A,B)\cdot |B| \right)^p = \left( \sum_{(A,B):A \neq \varnothing} \widetilde{\Xi}_{t_\alpha}(A,B)\cdot |B| \right)^p \cdot \mathds{1}\{Y_{t_\alpha}>0\}\]
and using the Cauchy-Schwarz inequality, the left-hand side of~\eqref{eq_bound_with_AB} is smaller than
\begin{align*}
\mathbb{E}\left[\left( \sum_{(A,B):A \neq \varnothing} \widetilde{\Xi}_{t_\alpha}(A,B)\cdot |B| \right)^{2p}\right]^{1/2}\cdot \mathbb{P}(Y_{t_\alpha} \ge 1)^{1/2}.
\end{align*}
	Using domination by a pure-birth process, the first term in the product above can be bounded by a finite constant that only depends on~$\lambda$,~$p$ and~$\alpha$. We now show that~$\mathbb{P}(Y_{t_\alpha} \ge 1)$ is smaller than~$C \alpha^{-2\mathrm{dist(o,u)}}$ for some~$C > 0$.

For~$i \in \{1,2\}$ and~$t \ge 0$, let~$K^{(i)}_t$ denote the set of vertices of~$\mathbb{T}^d$ that have been occupied by a type-$i$ particle in some herd for some time~$s \le t$, that is,
\begin{align*}
K^{(1)}_t:= \{v \in \mathbb{T}^d:\; \widetilde{\Xi}_s(A,B) > 0\text{ for some } s \le t \text{ and }(A,B) \text{ with } v \in A\},\\
K^{(2)}_t:= \{v \in \mathbb{T}^d:\; \widetilde{\Xi}_s(A,B) > 0\text{ for some } s \le t \text{ and }(A,B) \text{ with } v \in B\}.
\end{align*}
	We have that~$K^{(1)}_0 = \{o\}$,~$K^{(2)}_0 = \{u\}$ and~$t \mapsto K^{(1)}_t$ and~$t \mapsto K^{(2)}_t$ are both non-decreasing (with respect to set inclusion).  Moreover,~$K^{(1)}_t$ and~$K^{(2)}_t$ are connected subsets of~$\mathbb{T}^d$, since they only grow by the inclusion of vertices neighboring vertices that are already present.  Also note that as long as these sets stay disjoint, there can be at most one herd containing both types. In other words, letting
\[\sigma:= \inf\{t: K^{(1)}_t \cap K^{(2)}_t \neq \varnothing\},\]
we have, for any~$t \ge 0$,
\begin{equation}\label{eq_Y_witht}
\{\sigma > t\} \subseteq \{Y_s \le 1 \text{ for all } s \le t\}.
\end{equation}

Next, let
	\[\ell = \mathrm{dist}(o,u)\]
and let~$o = u_0 \sim u_1 \sim \ldots\sim u_\ell = u$ be the vertices of~$\mathbb{T}^d$ in the geodesic from~$o$ to~$u$. Let
	\[u':= u_{\lfloor \ell/3\rfloor},\quad \sigma^{(1)}:= \inf\{t \ge 0:\;u' \in K^{(1)}_t\}\]
and 
	\[u'':= u_{\lfloor 2\ell/3\rfloor},\quad \sigma^{(2)}:= \inf\{t \ge 0:\;u'' \in K^{(2)}_t\}.\]
It is easy to see that
\[\min(\sigma^{(1)},\sigma^{(2)}) \le \sigma.\]
Putting this together with~\eqref{eq_Y_witht}, we obtain
	\begin{align}\label{eq_monotone_Y2}
		\{Y_{s} \ge 2 \text{ for some } s \le t_\alpha\} \subseteq  \{\sigma \le t_\alpha\} \subseteq \{\min(\sigma^{(1)},\sigma^{(2)}) \le t_\alpha\},
	\end{align}
so we can bound
	\begin{align}
\nonumber \mathbb{P}(Y_{t_\alpha} \ge 1) &\stackrel{\eqref{eq_monotone_Y}}{=} \mathbb{P}(Y_s \ge 1 \text{ for all } 0 \le s \le t_\alpha) \\
		\label{eq_first_two_terms}&\stackrel{\eqref{eq_monotone_Y2}}{\le} \mathbb{P}(\sigma^{(1)} \le t_\alpha) +\mathbb{P}(\sigma^{(2)} \le t_\alpha) \\
		\label{eq_the_third} &\quad + \mathbb{P}(\min(\sigma^{(1)},\sigma^{(2)}) > t_\alpha,\; Y_{s} = 1 \text{ for } 0 \le s \le t_\alpha).
	\end{align}
		We now bound the three terms on the right-hand side separately.

	Let us first consider the probability in~\eqref{eq_the_third}. For any~$t$, if~$\min(\sigma^{(1)},\sigma^{(2)}) > t$ and~$Y_t = 1$, then a split in any of the edges
	\[\{u_{\lfloor \ell/3\rfloor}, u_{\lfloor \ell/3 \rfloor + 1}\},\; \{u_{\lfloor \ell/3\rfloor+1}, u_{\lfloor \ell/3 \rfloor + 2}\}, \; \ldots,\; \{u_{\lfloor 2\ell /3\rfloor - 1},\;u_{\lfloor 2\ell/3\rfloor} \}\]
	separates the two types permanently, causing~$Y$ to drop to zero. This observation gives
	\begin{align*}
		\mathbb{P}(\min(\sigma^{(1)},\sigma^{(2)}) > t_\alpha,\; Y_s = 1 \text{ for } 0 \le s \le t_\alpha) &\le \exp\{-\mathsf{v}\cdot t_\alpha \cdot (\lfloor 2\ell/3\rfloor - \lfloor \ell/3 \rfloor) \}\\
		&\le \alpha^{-2\ell},
	\end{align*}
	where the second inequality follows from the definition of~$t_\alpha$.

	We now turn to the two probabilities in~\eqref{eq_first_two_terms}. We only bound the first one; by a symmetry argument, the same bound will then apply to the second. Let~$(W_t)_{t \ge 0}$ be a growth process on~$(\mathbb{N}_0)^{\mathbb{T}^d}$ defined as follows. We let~$W_0(o) = 1$ and~$W_0(u) = 0$ for~$u \ne o$. We interpret~$W_t(v) = m$ as saying that there are~$m$ particles at~$v$ at time~$t$. Then, we define the dynamics by prescribing that independently, for any~$v \sim w$, a particle at~$v$ gives birth at a particle at~$w$ with rate~$\lambda$ (and particles never die). In particular,~$(\sum_v W_t(v))_{t \ge 0}$ is a pure-birth process in which the birth rate is~$d\lambda$. It is easy to see that the set-valued process~$(\{u\in \mathbb{T}^d:\;W_t(u) > 0\})_{t \ge 0}$ stochastically dominates~$(K^{(1)}_t)_{t \ge 0}$, and in particular,
	\begin{equation}\label{eq_bound_by_expectation}
	\mathbb{P}(\sigma^{(1)} \le t_\alpha) \le \mathbb{P}(W_{t_\alpha}(u') > 0) \le \mathbb{E}[W_{t_\alpha}(u')].
	\end{equation}

	 We now claim that, for any~$t \ge 0$ and~$v \in \mathbb{T}^d$,
	\begin{equation}\label{eq_ode_g}
		\mathbb{E}[W_t(v)] = e^{d\lambda t}\cdot p_t(v),
	\end{equation}
	where~$p_t(v) := \mathbb{P}(\mathcal{Z}_t = v)$, with~$(\mathcal{Z}_t)_{t \ge 0}$ the  continuous-time random walk on~$\mathbb{T}^d$ which starts at the root at time zero, and jumps from any vertex to any neighboring vertex with rate~$\lambda$. It is simple to verify~\eqref{eq_ode_g} using the observations that~$(t,v) \mapsto \mathbb{E}[W_t(v)]$ is the solution to
	\[\left\{ \begin{array}{l}\frac{\mathrm{d}}{\mathrm{d}t} g(t,v) = \lambda \sum_{w \sim v} g(t,w),\; t \ge 0,\; v \in \mathbb{T}^d,\\[.2cm]
	g(0,\cdot) = \delta_o(\cdot),\end{array}\right.\] and that the right-hand side of~\eqref{eq_ode_g} solves this equation, by direct computation.
	
	Putting together~\eqref{eq_bound_by_expectation} and~\eqref{eq_ode_g}, we have
	\[\mathbb{P}(\sigma^{(1)} \le t_\alpha) \le e^{d\lambda t_\alpha}\cdot p_{t_\alpha}(u') \le e^{d \lambda t_\alpha} \cdot \mathbb{P}(\mathrm{Poi}(d\lambda t_\alpha) > \lfloor \ell /3 \rfloor),\]
	where~$\mathrm{Poi}(d\lambda t_\alpha)$ represents a random variable with the Poisson distribution with parameter~$d\lambda t_\alpha$, which is the law of the number of jumps of~$(\mathcal{Z}_t)$ until time~$t_\alpha$. Since the tail of the Poisson distribution is lighter than exponential, the right-hand side above is smaller than~$C\alpha^{-2\ell}$ for some~$C > 0$, uniformly in~$\ell$. This concludes the proof.
\end{proof}

\begin{lemma}\label{lem_separate_two_again}
	For any~$p \ge 1$, there exists~$C_p' > 0$ (depending on~$\lambda$,~$\mathsf{v}$ and~$p$) such that the following holds. Fix~$B_0 \subseteq \mathbb{T}^d \backslash \{o\}$ and let~$(\widetilde{\Xi}_t)_{t \ge 0}$ be a two-type herds process with rates~$\lambda$,~$\mathsf{v}$ started from~$\widetilde{\Xi}_0 = \delta_{\{o\},B_0}$. Also let~$T_p := 12p\log(2d)/\mathsf{v}$. Then, (uniformly over the choice of~$B_0$),
	\begin{equation*}
		\E\left[\left(\sum_{(A,B):A \neq \varnothing} \widetilde{\Xi}_{T_p}(A,B)\cdot (|A|+|B|) \right)^p\; \right] \le C_p'.
	\end{equation*}
\end{lemma}
\begin{proof}
	Let
	\[{R} := \sum_{(A,B):A \neq \varnothing} \widetilde{\Xi}_{T_p}(A,B)\cdot (|A|+|B|);\]
	note that~${R}$ is the total number of particles in~$\widetilde{\Xi}_{T_p}$ that belong to herds that contain type-1 particles. 

	We enumerate~$\{o\} \cup B_0 = \{u_1,\ldots,u_{m}\}$, with~$u_1 = o$ and~$m = |B_0|+1$. We now define a multi-type herds process, as described in Remark~\ref{rem.ktype}. This new process, denoted~$(\doublewidetilde{\Xi}_t)_{t \ge 0}$, is taken in the same probability space that we have been considering, has rates~$\lambda$,~$\mathsf{v}$, and~$m$ types. It starts with a single herd, with a type-$1$ particle at~$u_1 = o$, a type-2 particle at~$u_2$,~$\ldots$, and a type~$m$ particle at~$u_{m}$. In analogy with~${R}$, we let ${R}'$ denote the total number of particles in $\doublewidetilde{\Xi}_{T_p}$ that belong to herds that contain type-1 particles. That is, if we use an~$m$-tuple~$(A_1,\ldots,A_{m})$ to represent a multi-type herd shape, then
	\[{R}' := \sum_{(A_1,\ldots,A_{m}): A_1 \neq \varnothing} \doublewidetilde{\Xi}_{T_p}(A_1,\ldots,A_m)\cdot (|A_1|+ \cdots + |A_m|).\]
	With similar reasoning as in the proof of Lemma~\ref{lem_projections}, we see that~${R}$ is stochastically dominated by~${R}'$; in particular,
	\[\E[{R}^p]^{1/p} \le \E[({R}')^p]^{1/p} \le \sum_{j=1}^m \E[(R'_j)^p]^{1/p},\]
	where
	\[R_j':= \sum_{(A_1,\ldots,A_{m}): A_1 \neq \varnothing} \doublewidetilde{\Xi}_{T_p}(A_1,\ldots,A_m)\cdot |A_j|.\]
	Note that~$R_1'$ is just the total number of type-1 particles in~$\doublewidetilde{\Xi}_{T_p}$. If we ignore all types except for type~$1$ in~($\doublewidetilde{\Xi}_t)$, we obtain a (one-type) herds process started from~$\delta_{\{o\}}$; hence,~$\mathbb{E}[(R'_1)^p] = \E[X_{T_p}^p] < \infty$ by Corollary~\ref{cor_explosion}. Next, for~$j \neq 1$, if we ignore all types except for types~$1$ and~$j$ in~$(\doublewidetilde{\Xi}_t)$, we see a two-type herds process started from~$\delta_{\{o\},\{u_j\}}$, and Lemma~\ref{lem_separate_two} (with~$\alpha = (2d)^p$) and the definition of~$T_p$ imply that~$\E[(R_j')^p] \le C_p (2d)^{-p\cdot \mathrm{dist}(o,u_j)}$. We then have
	\[\sum_{j=1}^m \E[(R_j')^p]^{1/p} \le \E[X_{T_p}^p]^{1/p} + C_p^{1/p} \sum_{j=1}^m (2d)^{-\mathrm{dist}(o,u_j)}.\]
	The second sum on the right-hand side is smaller than
	\[\sum_{u \in \mathbb{T}^d} (2d)^{-\mathrm{dist}(o,u)} \le \sum_{i=1}^\infty d^i\cdot (2d)^{-i} =1.\]
	We have thus proved that~$\E[R^p] \le (\E[X_{T_p}^p]^{1/p} + C_p^{1/p})^p$, so the proof is complete.
\end{proof}

\begin{lemma}\label{lem_last_in_chain}
	Let~$p \ge 1$ and let~$C_p'':= \max(C_p',\; \E[X_{T_p}^p])$, where~$C_p'$ and~$T_p$ are as in Lemma~\ref{lem_separate_two_again}. Then, for any~$\xi \in \mathcal{S}$, the herds process~$(\Xi_t)_{t \ge 0}$ with rates~$\lambda$,~$\mathsf{v}$ satisfies
	\begin{equation*}
		\P\left(\exists A:\; |A| \ge x, \;\Xi_{T_p}(A) > 0 \mid \Xi_0 = \xi \right) \le C_p''\cdot \frac{X(\xi)}{x^p},\quad x > 0.
	\end{equation*}
\end{lemma}
\begin{proof}
	By using Lemma~\ref{lem_decompose} and a union bound, it suffices to prove that for any~$A_0 \in P_\mathsf{f}(\mathbb{T}^d)$,
	\begin{equation}\label{eq_the_term}
		\P\left(\exists A:\; |A| \ge x, \;\Xi_{T_p}(A) > 0 \mid \Xi_0 = \delta_{A_0} \right) \le C_p''\cdot \frac{|A_0|}{x^p},\quad x > 0.
	\end{equation}
	We will prove this by induction on~$|A_0|$. For the case where~$|A_0| = 1$, we bound the left-hand side above by
	\begin{align*} \P(X_t \ge x \mid \Xi_0 = \delta_{\{o\}}) \le \frac{\E[X_{T_p}^p]}{x^p},
	\end{align*}
	by Markov's inequality.

	Now assume that~$|A_0| \ge 2$ and let~$u \in A_0$. Let~$B_0 := A_0 \backslash \{u\}$. We consider a two-type herds process~$(\widetilde{\Xi}_t)_{t \ge 0}$ with rates~$\lambda$,~$\mathsf{v}$ started from~$\delta_{\{u\},B_0}$.   Using Lemma~\ref{lem_projections}(a), the left-hand side of~\eqref{eq_the_term} is smaller than
	\begin{align*}
		& \P\left( \exists (A,B):\; |A|+|B| \ge x,\; \widetilde{\Xi}_{T_p}(A,B) > 0\right).
	\end{align*}
	In turn, this is smaller than
	\begin{align}
		\label{eq_numbers1}&\P\left( \exists (A,B):\;  A \neq \varnothing,\;|A|+|B| \ge x,\; \widetilde{\Xi}_{T_p}(A,B) > 0\right)\\
		\label{eq_numbers2}&+\P\left( \exists (A,B):\;  |B| \ge x,\; \widetilde{\Xi}_{T_p}(A,B) > 0\right).
	\end{align}
	The probability in~\eqref{eq_numbers1} is smaller than
	\[\P\left( \sum_{(A,B):A \neq \varnothing} \Xi_{T_p}(A,B)\cdot (|A|+|B|) \ge x \right) \le \frac{C_p'}{x^p},\]
	by  Markov's inequality and Lemma~\ref{lem_separate_two_again}. By Lemma~\ref{lem_projections}(b), the probability in~\eqref{eq_numbers2} can be expressed using a one-type herds process; it equals
	\[\P\left( \exists B:\; |B| \ge x,\; \Xi_{T_p}(B) > 0 \mid \Xi_0 = \delta_{B_0}\right),\]
	which is smaller than~$C_p''|B_0|/x^p$ by the induction hypothesis (since~$|B_0| = |A_0|-1$). We have then proved that
	\[\P\left(\exists A:\; |A| \ge x, \;\Xi_{T_p}(A) > 0 \mid \Xi_0 = A_0 \right) \le \frac{C_p'}{x^p} + \frac{C_p''(|A_0|-1)}{x^p} \le \frac{C_p''|A_0|}{x^p}.\]
\end{proof}

\begin{proof}[Proof of Lemma~\ref{lem_more_refin}]
	Fix~$s \ge T_p = 12p\log(2d)/\mathsf{v}$. Let~$E_1$ be the event that there is some herd at time~$s$ whose number of particles is larger than~$X_{s-T_p}^{\frac{1}{2p}}$, that is,
	\[E_1 := \left\{ \text{ there exists } A \in P_\mathsf{f}(\T^d) \text{ with } |A| \ge X_{s - T_p}^{\frac{1}{2p}} \text{ and } \Xi_s(A) > 0 \right\}.\]
	On~$E_1^c$, we bound
	\[\sum_{A \in P_\mathsf{f}(\mathbb{T}^d)} \Xi_s(A) \cdot |A|^p \le X_{s-T_p}^{1/2} \cdot \sum_{A \in P_\mathsf{f}(\T^d)} \Xi_s(A) \le X_{s-T_p}^{1/2} \cdot X_s.\]
	On~$E_1$, we bound
	\[\sum_{A \in P_\mathsf{f}(\mathbb{T}^d)} \Xi_s(A) \cdot |A|^p \le X_s^p.\]
	These bounds give
	\begin{equation}
		\label{eq_two_bounds_p}
		\E\left[\sum_{A \in P_\mathsf{f}(\mathbb{T}^d)} \Xi_s(A) \cdot |A|^p\right] \le \E[X_s \cdot X_{s-T_p}^{1/2}] + \E[X_s^p \cdot \mathds{1}_{E_1}].
	\end{equation}
We treat the two expectations on the right-hand side separately. For the first one, we start by bounding
	\begin{equation}
		\label{eq_aux_two_exps}
		\E[X_s \cdot X_{s-T_p}^{1/2}] = \E[\E[X_s \mid \mathcal{F}_{s-T_p}] \cdot X_{s-T_p}^{1/2}] \le \E[X_{s-T_p}^{3/2}] \cdot \E[X_{T_p}],
	\end{equation}
	where the inequality follows from~\eqref{eq_subadd0} and the Markov property. We now claim that there exists~$C > 0$ (not depending on~$s$) such that
	\begin{equation}
		\label{eq_aux_two_exps2}
		\E[X_{s-T_p}^{3/2}] \le C \E[X_s^{3/2}].
	\end{equation}
	To see this, first note that each particle that is alive at time~$s-T_p$ will stay alive until time~$s$ with probability~$q:= \mathrm{e}^{-T_p}$. Then, defining the event
	\[E_2 := \{X_s \ge \tfrac{q}{2} \cdot X_{s-T_p}\},\]
	we have, for any~$m \in \N$,
	\[\P(E_2 \mid \mathcal{F}_{s-T_p}) \cdot \mathds{1}\{X_{s-T_p} = m\} \ge  \P(\mathrm{Bin}(m,q) \ge \tfrac{q}{2}m)\]
	(and in case~$m = 0$, the left-hand side is trivially equal to 1). Hence,
	\[\P(E_2 \mid \mathcal{F}_{s-T_p}) \ge  \beta:= \inf_{m \ge 1} \P(\mathrm{Bin}(m,q) \ge \tfrac{q}{2} m) > 0.\]
	Then,
	\begin{align*}
		\E[X_s^{3/2}] \ge \E[X_s^{3/2} \cdot \mathds{1}_{E_2}] &\ge \tfrac{q}{2} \cdot \E[X_{s-T_p}^{3/2} \cdot \mathds{1}_{E'}]\\[.2cm]
		&= \tfrac{q}{2} \cdot \E[X_{s-T_p}^{3/2} \cdot \P(E_2 \mid \mathcal{F}_{s-T_p})] \ge \tfrac{q}{2}\cdot \beta \cdot \E[X_{s-T_p}^{3/2}],
	\end{align*}
	proving~\eqref{eq_aux_two_exps2} with~$C:=\frac{2}{\beta q}$.  With~\eqref{eq_aux_two_exps} and~\eqref{eq_aux_two_exps2}, and putting together all the terms that do not depend on~$s$ in a sufficiently large constant~$C$, we have proved that
	\[
		\E[X_s \cdot X_{s-T_p}^{1/2}] \le C\E[X_{s}^{3/2}].
	\]

	We now turn to the second term on the right-hand side of~\eqref{eq_two_bounds_p}. We first bound the conditional expectation given~$\mathcal{F}_{s-T_p}$ with Cauchy-Schwarz:
	\begin{align*}
		\E[X_s^p \cdot \mathds{1}_{E_1} \mid \mathcal{F}_{s-T_p}] &\le \E[X_s^{2p}\mid \mathcal{F}_{s-T_p}]^{1/2} \cdot \P(E_1 \mid \mathcal{F}_{s-T_p})^{1/2}.
	\end{align*}
	Using~\eqref{eq_subadd0} and the Markov property, 
	\[\E[X_s^{2p}\mid \mathcal{F}_{s-T_p}] \le X_{s-T_p}^{2p}\cdot \E[X_{T_p}^{2p}].\]
	Using Lemma~\ref{lem_last_in_chain}  (and again the Markov property), we have
	\[\mathbb{P}(E_1 \mid \mathcal{F}_{s-T_p}) \le C_{4p^2}''\cdot \frac{X_{s-T_p}}{(X_{s-T_p}^{1/(2p)})^{4p^2}} = C_{4p^2}'' \cdot X_{s-T_p}^{1-2p}.\]
	Then,
	\[\E[X_s^p \cdot \mathds{1}_{E_1} \mid \mathcal{F}_{s-T_p}] \le C_{4p^2}'' \cdot \E[X_{T_p}^{2p}] \cdot  X_{s-T_p},\]
	so
	\[\E[X_s^p \cdot \mathds{1}_{E_1} ] \le C_{4p^2}'' \cdot \E[X_{T_p}^{2p}] \cdot \E[X_{s-T_p}] \le C_{4p^2}'' \cdot \E[X_{T_p}^{2p}] \cdot C\varphi^{s-T_p},\]
	where the second inequality follows from~\eqref{eq_const}. Putting together all the constants that do not depend on~$s$, this gives
	\[\E[X_s^p \cdot \mathds{1}_{E_1} ] \le C\varphi^{s}.\]
	%Since we are assuming that~$\lambda < \bar{\lambda}(\mathsf{v})$, we have~$\varphi < 1$ by Lemma~\ref{lem_surv_cond} and Proposition~\ref{prop_strict_mon}. 
\end{proof}

\section{Extinction time on a random $d$-regular graph with switching} \label{s_finite_graph}

The goal of this section is to prove Theorem~\ref{theo.1}. As mentioned in the Introduction, the first part of the theorem was already proved in~\cite{da2021contact}, so we only prove the second part here.

This section is organized as follows. In Section~\ref{ss_prelim_graph}, we give a detailed construction of the dynamic random graph and the contact process which co-evolves on this graph. After doing so, we argue that a version of the usual self-duality relation of the contact process is satisfied here. Due to this relation, in proving quick extinction, it suffices to study the process started from a single infection. We start this study in Section~\ref{ss_exploration}, where we introduce an exploration process, which reveals the contact process but only reveals partial information about the graph.  In Section~\ref{ss_Markov_chain}, we prove a general Markov chain lemma which allows us to couple this exploration process with (a projection of) the herds process. Finally, in Section~\ref{ss_concluding_proof}, we take advantage of this coupling to give the proof of Theorem~\ref{ss_concluding_proof}.

\subsection{Preliminaries: dynamic graph, joint evolution, duality}\label{ss_prelim_graph}

Let us define the class of graphs in which our dynamic graph process takes values. Fix~$n \in \N$ and let~$V_n := [n] = \{1,\ldots,n\}$. The set~$\{(u,a):\;u \in V_n,\: a \in \{1,\ldots,d\}\}$ is called the \textit{set of half-edges}. Given a perfect matching of the set of half-edges (that is, a bijection~$\sigma$ from this set to itself with no fixed points and equal to its own inverse), we can obtain a (multi-)graph by setting~$V_n$ as the set of vertices and prescribing that each pair~$\{(u,a),(u',a')\}$ with~$(u',a') = \sigma(u,a)$ corresponds to an edge between~$u$ and~$u'$. Let~$\mathcal{G}_n$ denote the set of all multi-graphs that can be obtained in this way. Deterministic elements of~$\mathcal{G}_n$ will typically be denoted by~$\mathsf{g}$, whereas random elements of~$\mathcal{G}_n$ will be denoted by~$G$ or~$G_t$ (in the case of a process).

Fix~$\mathsf{g} \in \mathcal{G}_n$. A \textit{switch code} for~$\mathsf{g}$ is a triple~$\mathsf{m}=(e_1,e_2,\eta)$, where~$e_1,e_2$ are distinct edges of~$\mathsf{g}$ and~$\eta \in \{-,+\}$. Fix such a switch code, with~$e_1 = \{(u,a),(u',a')\}$ and~$e_2 = \{(v,b),(v',b')\}$ so that~$(u,a)< (u',a')$ and~$(v,b) < (v',b')$ in lexicographic order. In case~$\eta = +$, we let~$\Gamma^{\mathsf{m}}(\mathsf{g})$ be the graph obtained from~$\mathsf{g}$ by replacing~$e_1$ and~$e_2$ by the edges $\{(u,a),(v,b)\}$ and $\{(u',a'),(v',b')\}$ (and keeping all other edges intact). In case~$\eta=-$, we instead replace~$e_1,e_2$ by~$\{(u,a),(v',b')\}$ and~$\{(u',a'),(v,b)\}$. 

The random graph process~$(G_t)_{t \ge 0}$ is the continuous-time Markov chain on~$\mathcal{G}_n$ which jumps from~$\mathsf{g}$ to~$\mathsf{g}'$ with rate~$\upupsilon_n:= \frac{\mathsf{v}}{nd}$ if~$\mathsf{g}' = \Gamma^{\mathsf{m}}(\mathsf{g})$ for some switch code~$\mathsf{m}$ of~$\mathsf{g}$ (and rate 0 otherwise). This chain is reversible with respect to the uniform measure on~$\mathcal{G}_n$. We will always start the graph dynamics from this distribution.

We now fix~$\lambda > 0$ and define the contact process~$(\xi_t)_{t \ge 0}$ with infection rate~$\lambda$ on the dynamic graph~$(G_t)$. Although we could do so by describing the jump rates of the joint Markov chain~$(G_t,\xi_t)_{t \ge 0}$, we will instead use a Poisson graphical construction.  We take a probability space with probability measure~$\mathbb{P}$ in which the process~$(G_t)$ is defined, and also (independently of~$(G_t)$), the following Poisson point processes (all independent) are defined:
\begin{itemize}
\item for each vertex~$u$, a Poisson point process~$R^u$ on~$[0,\infty)$ with intensity 1 (recovery times);
\item for each half-edge~$(u,a)$, a Poisson point process~$R^{(u,a)}$ with intensity~$\lambda$ (transmission times).
\end{itemize}
 Naturally, when~$t \in R^u$, vertex~$u$ goes to state 0 (if not already there) at time~$t$. Moreover, when~$t \in R^{(u,a)}$, there is a transmission from~$u$ to the vertex~$v$ that owns the half-edge to which~$(u,a)$ is matched in~$G_t$ (so that, if~$u$ was in state 1 just before~$t$, then~$v$ goes to state 1, if not already there, at~$t$). For each~$A \subset V_n$, we let~$(\xi^A_t)_{t \ge 0}$ be the contact process on~$(G_t)$ with~$\xi^A_0 = \mathds{1}_A$ and obtained from this graphical construction (as usual, the graphical construction gives us contact processes started from all possible initial configurations, all coupled in a single probability space and respecting the monotonicity of set inclusion).

The usual duality relation
\begin{equation}
\label{eq_duality_relation}
\mathbb{P}(\xi^A_t \cap B \neq \varnothing) = \mathbb{P}(\xi^B_t \cap A \neq \varnothing) \quad \text{for all }t \ge 0,\;A, B \subset V_n
\end{equation}
holds in this context, but it is important to note that the above probabilities are annealed in the graph environment. Let us briefly prove~\eqref{eq_duality_relation}. Fix~$t \ge 0$ and~$A,B \subset V_n$. Letting~$(\mathsf{g}_s)_{0 \le s \le t}$ be a possible realization of the trajectory of~$(G_s)_{0 \le s \le t}$, we have
\begin{align*}
&\mathbb{P}(\xi^A_t \cap B \neq \varnothing \mid (G_s)_{0 \le s \le t} = (\mathsf{g}_s)_{0 \le s \le t}) \\
&\quad = \mathbb{P}(\xi^B_t \cap A \neq \varnothing \mid (G_s)_{0 \le s \le t} = (\mathsf{g}_{t-s})_{0 \le s \le t}).
\end{align*}
This is verified using a standard argument involving infection paths and time-reversibility of Poisson processes. Integrating this equality over the choice of~$(\mathsf{g}_s)_{0 \le s \le t}$ and using the fact that~$(G_s)_{0 \le s \le t}$ has the same law as~$(G_{t-s})_{0 \le s \le t}$ gives~\eqref{eq_duality_relation}.

Letting~$\bar{u} \in V_n$ be arbitrary (and deterministic) and writing~$\xi^{\bar{u}}_t$ instead of~$\xi^{\{\bar{u}\}}_t$, we have
\begin{equation}
\label{eq_apply_duality}
\P(\xi^{V_n}_t \neq \varnothing) \le \sum_{u \in V_n}\P(\xi^{V_n}_t(u) = 1) = n\cdot \P(\xi^{V_n}_t(\bar{u}) = 1) = n \cdot \mathbb{P}(\xi^{\bar{u}}_t \neq \varnothing),
\end{equation}
where the equalities follow from symmetry and duality, respectively. Due to~\eqref{eq_apply_duality}, the analysis of the extinction time of the contact process started from all vertices infected can be reduced to the analysis of the extinction time of the contact process from a single infection at the arbitrary vertex~$\bar{u}$. For the rest of this section,~$\bar{u}$ remains fixed.

Next, we describe an \textit{exploration process}, which only reveals partial information about~$(G_t)_{t \ge 0}$, namely, only the matching of half-edges at certain points in time on a need-to-know basis imposed by the transmission times of the contact process.

\subsection{Exploration process} \label{ss_exploration}
Let
\begin{equation*}
\mathcal{P}:= \{ \{(u,a),(u',a')\}:\; u, u' \in V_n,\; a,a' \in \{1,\ldots, d\},\; (u,a) \neq (u',a')\}
\end{equation*}
be the set of all potential edges of our random graph.
A set~$\mathcal{E} \subset \mathcal{P}$ is called \textit{independent} if any two elements of~$\mathcal{E}$ have no half-edge in common. Let
\begin{equation}\label{eq_def_of_scrP}
\mathscr{P}:= \{\mathcal{E} \subset \mathcal{P}:\; \mathcal{E} \text{ is independent}\}.
\end{equation}
In the same probability space where~$(G_t)$ and the graphical construction of the contact process are defined, we now define a process~$(\mathscr{E}_t)_{t \ge 0}$ taking values in~$\mathscr{P}$. Intuitively,~$\mathscr{E}_t$ represents a set of edges that are known to be part of~$G_t$, having been revealed by an exploration induced by the contact process activity. This process will have the following features:
\begin{itemize}
\item[(P1)] it starts from~$\mathscr{E}_0 = \varnothing$;
\item[(P1)] for any~$t$, every element of~$\mathscr{E}_t$ is an edge of~$G_t$;
\item[(P3)] the pair~$(\xi^{\bar{u}}_t,\mathscr{E}_t)_{t \ge 0}$ is a Markov chain;
\item[(P4)] for any~$t$, conditionally on~$(\xi^{\bar{u}}_t, \mathscr{E}_t)$, the distribution of the edges of~$G_t$ apart from those in~$\mathscr{E}_t$ is uniform. More precisely, the pairing in~$G_t$ of the half-edges in the set~$\{(u,a):(u,a) \text{ not in any edge of }\mathscr{E}_t\}$ is uniformly distributed among all possibilities.
\end{itemize}

In order to define the exploration, we will need auxiliary times. Let~$T_1$ be the first time in which there is a transmission mark at a half-edge emanating from~$\bar{u}$; let~$\mathscr{E}_t = \varnothing$ for all~$t \in [0,T_1)$.  In case~$T_1 < \infty$, say, due to a transmission mark at the half-edge~$(\bar{u},a)$, we reveal the half-edge~$(v,b)$ to which~$(\bar{u},a)$ is paired in~$G_{T_1}$, and include the edge~$\{(\bar{u},a),(v,b)\}$ in~$\mathscr{E}_{T_1}$.

Now assume that we have already defined stopping times (with respect to the filtration of~$(G_t)$ and the graphical construction)~$T_1 \le T_2 \le \cdots  \le T_k$, and that we have defined~$\mathscr{E}_t$ for~$0 \le t \le T_k$. In case~$T_k = \infty$, set~$T_{k+1} = \infty$; from now on, assume that~$\{T_k < \infty\}$ occurs. Let~$T_{k+1}$ be the first time~$t > T_k$ when:
\begin{itemize}
\item either a switch occurs involving at least one edge of~$\mathscr{E}_{T_k}$ (call this \textit{Case~1}),
\item or ``the contact process tries to use an unexplored edge'', that is, a transmission mark appears at a half-edge emanating from some vertex of~$\xi_t$, and this half-edge is not part of an edge of~$\mathscr{E}_{T_k}$ (\textit{Case 2});
\end{itemize}
We set~$\mathscr{E}_t = \mathscr{E}_{T_k}$ for~$t \in (T_k, T_{k+1})$, and~$\mathscr{E}_{T_{k+1}}$ is defined as follows. 
\begin{itemize}
\item In Case 1, there are two sub-cases. First assume that the switch at time~$T_{k+1}$ involves an edge~$e$ of~$\mathscr{E}_{T_k}$ and another edge outside~$\mathscr{E}_{T_k}$. Then, we let~$\mathscr{E}_{T_{k+1}} = \mathscr{E}_{T_k} \backslash \{e\}$. Now assume that the switch at time~$T_{k+1}$ involves two edges~$e,e'$ of~$\mathscr{E}_{T_k}$, transforming them into the two new edges~$e'',e'''$. We then set~$\mathscr{E}_{T_{k+1}} = (\mathscr{E}_{T_k} \backslash \{e,e'\}) \cup \{e'',e'''\}$.
\item In Case 2, we reveal the half-edge that is matched at time~$t$ to the half-edge having the transmission mark at that time; letting~$e$ be the edge formed by these two half-edges, we let~$\mathscr{E}_{T_{k+1}} = \mathscr{E}_{T_k} \cup \{e\}$.
\end{itemize}
This completes the description of the exploration, and it should be clear that properties (P1), (P2), (P3) and (P4) listed earlier are indeed satisfied.

Our next step is to use the exploration process as a tool to couple the contact process~$(\xi^{\bar{u}}_t)$ with a herds process. 

\subsection{A Markov chain lemma and its application} \label{ss_Markov_chain}
We now prove a general result about coupling two continuous-time Markov chains.

\begin{lemma}
\label{lem_coupling_two_MCs}
Let~$\mathcal{X}_1$ and~$\mathcal{X}_2$ be countable sets, and let~$r_1: \mathcal{X}_1 \times \mathcal{X}_1 \to [0,\infty)$ and~$r_2: \mathcal{X}_2 \times \mathcal{X}_2 \to [0,\infty)$ be functions defining the jump rates for continuous-time (non-explosive) Markov chains on~$\mathcal{X}_1$ and~$\mathcal{X}_2$, respectively. Assume that there is a subset~$\mathcal{X}_1' \subseteq \mathcal{X}_1$ and functions~$\Psi: \mathcal{X}_1' \to \mathcal{X}_2$ and~$f: \mathcal{X}_1' \to [0,\infty)$ such that the following two conditions hold:
	\begin{equation}\label{eq_1_coupling_ineq}
\sum_{y \in \mathcal{X}_1 \backslash \mathcal{X}_1'} r_1(x,y) \le f(x) \quad \text{for all } x \in \mathcal{X}_1'
\end{equation}
and
\begin{equation}\label{eq_2_coupling_ineq}
	\sum_{\substack{z \in \mathcal{X}_2,\\ z \neq \Psi(x)}} \left| r_2(\Psi(x),z) - \sum_{y \in \Psi^{-1}(z)} r_1(x,y)   \right| \le f(x)\quad \text{for all } x \in \mathcal{X}_1'.
\end{equation}
Fix~$\bar{x} \in \mathcal{X}_1'$. Then, there exists a coupling~$(\mathcal{A}_t,\mathcal{B}_t)_{t \ge 0}$ on~$\mathcal{X}_1 \times \mathcal{X}_2$ with the following properties:
\begin{itemize}
\item[\textnormal{(a)}] $\mathcal{A}_0 = \bar{x}$ and~$(\mathcal{A}_t)_{t \ge 0}$ is a Markov chain on~$\mathcal{X}_1$ with jump rates~$r_1$;
\item[\textnormal{(b)}] $\mathcal{B}_0 = \Psi(\bar{x})$ and~$(\mathcal{B}_t)_{t \ge 0}$ is a Markov chain on~$\mathcal{X}_2$ with jump rates~$r_2$;
\item[\textnormal{(c)}] letting
\[\sigma:= \inf\{t: \mathcal{B}_t \neq \Psi(\mathcal{A}_t)\}\]
and, for any~$a > 0$,
\[T_a := \inf\{t: f(\mathcal{A}_t) > a\},\]
we have, for any~$t  > 0$,
\begin{equation*}
	\P(\sigma \le t \wedge T_a) \le 2a t.
\end{equation*}
\end{itemize}
\end{lemma}

\begin{proof}
	We will define a continuous-time Markov chain~$(\mathcal{A}_t,\mathcal{B}_t,\mathcal{W}_t)_{t \ge 0}$ taking values in the set
	\begin{equation}
		\label{eq_two_sets}
	\{(x,\Psi(x),1): \; x \in \mathcal{X}_1'\} \cup \{(x,y,0):\;x \in \mathcal{X}_1,\;y\in \mathcal{X}_2\},
	\end{equation}
	starting from~$(\bar{x},\Psi(\bar{x}),1)$. The pair~$(\mathcal{A}_t,\mathcal{B}_t)$ will satisfy the properties in the statement. The third coordinate process~$(\mathcal{W}_t)$ will be a non-decreasing process (it jumps at most once, from 1 to 0) with the property that for all~$t \le \inf\{s: \mathcal{W}_s = 0\}$, we have~$\mathcal{B}_t = \Psi(\mathcal{A}_t)$. So, we interpret~$\mathcal{W}_t$ as the indicator of the event that ``the coupling still works at time~$t$''.

	In order to define this chain, we need to specify the jump rates. When the third coordinate equals zero (meaning that the coupling is already broken), the first and second coordinates move independently, according to the chains defined by~$r_1$ (on~$\mathcal{X}_1$) and~$r_2$ (on~$\mathcal{X}_2$), respectively. More precisely, from any triple of the form~$(x,y,0)$, the chain jumps as follows:
	\begin{itemize}
		\item for each~$x' \in \mathcal{X}_1$, it jumps to~$(x',y,0)$ with rate~$r_1(x,x')$;
		\item for each~$y' \in \mathcal{X}_2$, it jumps to~$(x,y',0)$ with rate~$r_2(y,y')$.
	\end{itemize}

	We now need to specify the jump rates from points in the first set in the union in~\eqref{eq_two_sets}. In order to do so, we first introduce some notation. For each~$x \in \mathcal{X}_1'$, we let
	\[[x] := \Psi^{-1}(\Psi(x)) \subset \mathcal{X}_1'.\]
	For each~$x \in \mathcal{X}_1$ and each~$S \subset \mathcal{X}_1$, we write
	\[r_1(x,S):= \sum_{y \in S} r_1(x,y).\] 

	Now, fix~$x \in \mathcal{X}_1'$. The following list describes all the possible jumps that the chain can take from~$(x,\Psi(x),1)$ (this starting location is kept fixed throughout the list), and their respective rates:
	\begin{itemize}
		\item \textit{$\mathcal{A}$ and $\mathcal{B}$ jump together, stay coupled:} for each~$y \in \mathcal{X}_1' \backslash [x]$, jump to~$(y,\Psi(y),1)$ with rate
			\[r_1(x,y) \cdot \frac{r_1(x,[y])\wedge r_2(\Psi(x),\Psi(y))}{r_1(x,[y])};\]
		\item \textit{$\mathcal{A}$ jumps alone inside $[x]$, stay coupled:} for each~$y \in [x]\backslash \{x\}$, jump to $(y,\Psi(y),1) = (y,\Psi(x),1)$ with rate~$r_1(x,y)$;
		\item \textit{$\mathcal{A}$ jumps alone leaving $\mathcal{X}_1'$, break coupling:} for each~$y \in \mathcal{X}_1 \backslash \mathcal{X}_1'$, jump to~$(y,\Psi(x),0)$ with rate~$r_1(x,y)$;
		\item \textit{$\mathcal{A}$ jumps alone inside $\mathcal{X}_1'$, break coupling:} for each~$y \in \mathcal{X}_1' \backslash [x]$, jump to~$(y,\Psi(x),0)$ with rate
			\[r_1(x,y) \cdot  \left(1- \frac{r_1(x,[y])\wedge r_2(\Psi(x),\Psi(y))}{r_1(x,[y])}\right);\]
		\item \textit{$\mathcal{B}$ jumps alone, breaks coupling:} for each~$z \in \mathcal{X}_2$, jump to~$(x,z,0)$ with rate
			\[r_2(\Psi(x),z) - (r_1(x,\Psi^{-1}(z))\wedge r_2(\Psi(x),z)).\]
	\end{itemize}
	It is straightforward to check that the marginal rates for~$(\mathcal{A}_t)$ and~$(\mathcal{B}_t)$ are correct, so that items (a) and (b) in the statement of the lemma hold.
	
	For each~$x \in \mathcal{X}_1'$, let~$\mathcal{R}(x)$ denote the rate at which~$(\mathcal{A}_t,\mathcal{B}_t,\mathcal{W}_t)$ jumps from~$(x,\Psi(x),1)$ to the set~$\mathcal{X}_1 \times \mathcal{X}_2 \times \{0\}$, where the coupling is broken. From the above rates, and then using~\eqref{eq_1_coupling_ineq} and~\eqref{eq_2_coupling_ineq}, it can be seen that
	\begin{equation}
		\label{eq_bound_R}
	\mathcal{R}(x) = r_1(x,\mathcal{X}_1\backslash \mathcal{X}_1') + \sum_{\substack{z \in \mathcal{X}_2,\\ z \neq \Psi(x)}} |r_1(x,\Psi^{-1}(z)) -r_2(\Psi(x),z) | \le 2 f(x). 
	\end{equation}
Next, let~$\sigma' := \inf\{t: \mathcal{W}_t = 0\}$, and recall that~$T_a:= \inf\{t: f(\mathcal{A}_t) > a\}$. The process
	\[M_t:= \mathds{1}\{ \sigma' \le t \wedge T_a\} - \int_0^{t \wedge \sigma' \wedge T_a} \mathcal{R}(\mathcal{A}_s)\;\mathrm{d}s,\qquad t \ge 0\]
	is easily seen to be a martingale. Then, for any~$t \ge 0$,
	\begin{align*}
		0 = M_0 = \mathbb{E}[M_t] &= \mathbb{P}(\sigma' \le t \wedge T_a) - \mathbb{E}\left[ \int_0^{t \wedge \sigma' \wedge T_a} \mathcal{R}(\mathcal{A}_s)\;\mathrm{d}s\right] \\[.2cm]&\stackrel{\eqref{eq_bound_R}}{\ge} \mathbb{P}(\sigma' \le t \wedge T_a) - 2at.
	\end{align*}
 	Now, recalling that~$\sigma:= \inf\{t: \mathcal{B}_t \neq \Psi(\mathcal{A}_t)\}$, we have that~$\sigma' \le \sigma$, so
	\[\P(\sigma \le t \wedge T_a) \le \P(\sigma' \le t \wedge T_a) \le 2at.\]
\end{proof}

In the application we have in mind for this lemma, the first Markov chain is the pair~$(\xi_t,\mathscr{E}_t)$ consisting of the contact process and the exploration process in the random dynamic graph, as described in the previous subsection (recall that this pair is a Markov chain). The second Markov chain is a certain function of the herds process. We will need to give some definitions for both, as well as for the mapping~$\Psi$ between them.

\subsubsection{First Markov chain: contact and exploration process}
The process~$(\xi_t,\mathscr{E}_t)_{t \ge 0}$ (contact process and exploration process on the dynamic random graph~$(G_t)$) takes values in the state space
\begin{equation*}
	\label{eq_X1_for_application}
	\mathcal{X}_1 := \{(A,\mathcal{E}):\; A \subset [n],\; \mathcal{E} \in \mathscr{P}\},
\end{equation*}
where we recall the definition of~$\mathscr{P}$ in~\eqref{eq_def_of_scrP}. We denote by~$r_1(\cdot,\cdot)$ the function giving the jump rates of this chain.

Given~$(A,\mathcal{E}) \in \mathcal{X}_1$, we define the \textbf{graph induced by~$(A,\mathcal{E})$}, denoted by~$\mathrm{Graph}(A,\mathcal{E})$, as follows. First enumerate~$\mathcal{E} = \{e_1,\ldots,e_m\}$, with
\[e_1=  \{(u_1,a_1),(u_1',a_1')\},\quad \ldots,\quad e_m = \{(u_m,a_m),(u_m',a_m')\}.\]
Then,~$\mathrm{Graph}(A,\mathcal{E})$ is the graph with vertex set~$A \cup \{u_1,u_1',\ldots, u_m,u_m'\}$ and edge set~$\mathcal{E}$.

\subsubsection{Second Markov chain: herds process modulo automorphisms}
Recall the definition of the set~$P_\mathsf{f}(\mathbb{T}^d)$ of herd shapes from Definition~\ref{def_shapes}. 
For~$A \in P_\mathsf{f}(\mathbb{T}^d)$, define
\begin{equation*}
	[A] := \left\{\begin{array}{r} A' \in P_\mathsf{f}(\mathbb{T}^d):\; \text{there is a graph isomorphism $\psi:\mathbb{T}^d \to \mathbb{T}^d$} \\ \text{such that } \psi(A) = A'\end{array}\right\}.
\end{equation*}
This decomposes~$P_\mathsf{f}(\mathbb{T}^d)$ into equivalence classes.

Recall the definition of the set~$\mathcal{S}$ of herd configurations~$\mathcal{S}$ from Definition~\ref{def_configurations}. Given~$\xi \in \mathcal{S}$, define~$[\xi]:\{[A]:A \in P_\mathsf{f}(\mathbb{T}^d)\}\to \mathbb{N}_0$ by setting
\[[\xi]([A]) = \sum_{A' \in [A]} \xi(A'). \]
We then define
\begin{equation*}
	\label{eq_X2_for_application}
\mathcal{X}_2:= \{[\xi]: \xi \in \mathcal{S}\},
\end{equation*}
the \textbf{set of herd configurations modulo automorphisms}. Letting~$(\Xi_t)_{t \ge 0}$ be the herds process, we note that, by Lemma~\ref{lem_auto}, the process~$([\Xi_t])_{t \ge 0}$ is a Markov chain on~$\mathcal{X}_2$. We let~$r_2(\cdot,\cdot)$ denote the function giving the jump rates of this chain.

\subsubsection{The mapping~$\Psi$ and the error bound~$f$}
Now  that we have defined the pairs~$(\mathcal{X}_1,r_1)$ and~$(\mathcal{X}_2,r_2)$ that we will use in our application of Lemma~\ref{lem_coupling_two_MCs}, we will also define the sets~$\mathcal{X}_1'$ and the functions~$\Psi:\mathcal{X}_1' \to \mathcal{X}_2$ and~$f: \mathcal{X}_1' \to [0,\infty)$ that appear in the assumptions of that lemma. 

We start with
\begin{equation}\label{eq_choice_of_X_prime}
\mathcal{X}_1':= \{(A,\mathcal{E}) \in \mathcal{X}_1:\; \mathrm{Graph}(A,\mathcal{E}) \text{ is a forest}\}.   
\end{equation}

The mapping~$\Psi$ is easy to understand (Figure~\ref{fig_mapping}  provides an instant explanation) but somewhat clumsy to define. Fix $(A,\mathcal{E}) \in \mathcal{X}_1'$. Let~$\mathscr{C}_1,\ldots,\mathscr{C}_m$ be the connected components of~$\mathrm{Graph}(A,\mathcal{E})$ that contain at least one vertex of~$A$. For~$i \in \{1,\ldots, m\}$, let~$A_i$ be the set of vertices of~$A$ that intersect~$\mathscr{C}_i$. Since~$\mathscr{C}_i$ is a tree in which all vertices have degree at most~$d$, there exists an isomorphism~$\psi_i$ between~$\mathscr{C}_i$ and some connected subgraph of~$\mathbb{T}^d$ (in fact there are infinitely many such isomorphisms, but we choose one in some arbitrary way). Then,~$\xi(A,\mathcal{E}):=\sum_{i=1}^m \delta_{\psi_i(A_i)}$ is a herd configuration, and we let
\[\Psi(A,\mathcal{E}):=[\xi] \in \mathcal{X}_2.\]

\begin{figure}[H]
\begin{center}
\setlength\fboxsep{0cm}
\setlength\fboxrule{0.01cm}
\fbox{\includegraphics[width=0.7\textwidth]{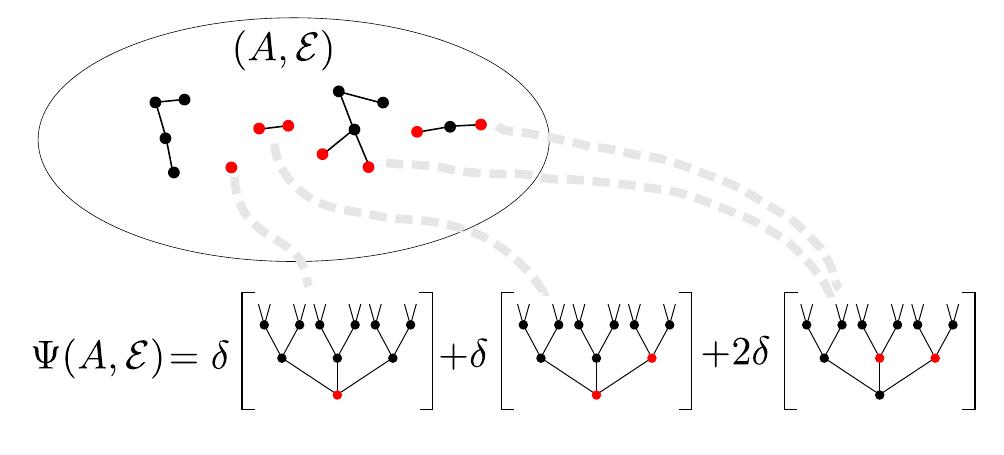}}
\end{center}
	\caption{\label{fig_mapping} Illustration of the mapping $\Psi$. Above, the pair $(A,\mathcal{E})$ is depicted (red vertices are those that belong to $A$, and black vertices are those that belong to $\mathrm{Graph}(A,\mathcal{E})$ but not to $A$). Below, the herd configuration modulo automorphisms $\Psi(A,\mathcal{E})$ is shown. Note that one of the connected components of $\mathrm{Graph}(A,\mathcal{E})$ has no counterpart in~$\Psi(A,\mathcal{E})$, because none of its vertices belongs to $A$.}
\end{figure}

It is now straightforward to verify that there exists a constant~$C_f > 0$ such that conditions~\eqref{eq_1_coupling_ineq} and~\eqref{eq_2_coupling_ineq} are satisfied with the choice
\[f(A,\mathcal{E}) := C_f\frac{(|A|+|\mathcal{E}|)^2}{n}.\]
We omit the details.

We now have all the ingredients to apply Lemma~\ref{lem_coupling_two_MCs}. Given an initial condition~$(A,\mathcal{E}) \in \mathcal{X}_1'$ for the exploration process (which will often, but not always, be equal to~$(\{\bar{u}\},\varnothing)$), we can obtain the coupling~$(\mathcal{A}_t,\mathcal{B}_t)_{t \ge 0}$ started from~$((A,\mathcal{E}),\Psi(A,\mathcal{E}))$ and satisfying the properties guaranteed by that lemma.

\subsection{Proof of Theorem~\ref{theo.1}} \label{ss_concluding_proof}
For the rest of this section, we assume that~$\lambda< \bar{\lambda}(\mathsf{v})$.
By~\eqref{eq_apply_duality}, it suffices to prove that, for~$C > 0$ large enough, we have
\begin{equation*}
    n \cdot \P(\xi^{\bar{u}}_{C \log n} \neq \varnothing) \xrightarrow{n \to \infty} 0.
\end{equation*}

Let us explain our strategy to prove this. We take advantage of the coupling with the herds process from the previous section. 
The probability that the contact process started from~$\{\bar{u}\}$ survives until time~$C \log n$ (with~$C$ some large constant), and moreover the coupling remains good (meaning that~$\mathcal{B}_t = \Psi(\mathcal{A}_t)$) for time~$C' \log n$ (with~$C' < C$ but still large) is~$o(1/n)$, since this would imply survival of the herds process, which is subcritical, until~$C' \log n$. However, the probability that the coupling turns bad before extinction is not~$o(1/n)$, so we have to deal with that event. The most problematic case is that the coupling turns bad  due to the exploration process finding an edge that causes the explored graph to no longer be a forest. In that case, apart from events of probability~$o(1/n)$,  this problematic edge is deleted after a short amount of time due to a switch (with no other problematic edges appearing in the meantime), and the explored region goes back to being a forest. At this moment when a forest reappears, we can start a brand new coupling between the exploration process (starting from its current state~$(\xi_t,\mathscr{E}_t)$) and a herds process (starting from~$\Psi(\xi_t,\mathscr{E}_t)$). Now, this second  coupling also turning bad has too low probability (when we consider this event together with the already low probability of the breaking of the first attempt). It is also unlikely that this second coupling stays active for a long time without turning bad, for the same reason as for the first one.

The above explanation shows that the argument is naturally structured in three stages (of course, not all of them necessarily occur): the first coupling attempt, then the period until a problematic edge is removed, and then the second coupling attempt. We encapsulate Stages~2 and~3 in two lemmas, in reverse order: Lemma~\ref{lem_stage_3} below deals with Stage 3, and Lemma~\ref{lem_from_not_forest} with Stage 2. Having these two lemmas in place, we are able to tell the full story from the beginning of Stage 1, concluding the proof.

\begin{lemma}\label{lem_stage_3}
    Let~$(A,\mathcal{E}) \in \mathcal{X}_1'$, where~$\mathcal{X}_1'$ is as in~\eqref{eq_choice_of_X_prime}. Assume that~$|A| + |\mathcal{E}| \le n^{1/{6}}$. Let~$(\xi_t,\mathscr{E}_t)_{t \ge 0}$ be a contact process and exploration process started from~$(\xi_0,\mathscr{E}_0) = (A,\mathcal{E})$. Then, letting~$\tau$ denote the extinction time of the contact process, and~$\varphi = \varphi(\lambda, \mathsf{v})$ be the growth index of the herds process (as in~\eqref{def_varphi}), for~$n$ large enough we have
\begin{equation}
    \label{eq_tau_using_coupling}
    \P\left(\tau > \frac{2}{|\log \varphi|} \log n \right)\le \frac{1}{\sqrt{n}}.    
\end{equation}

\end{lemma}
\begin{proof}  
Let~$(\mathcal{A}_t,\mathcal{B}_t)_{t \ge 0}$ be the coupling obtained from Lemma~\ref{lem_coupling_two_MCs}, started from~$(\mathcal{A}_0,\mathcal{B}_0) = ((A,\mathcal{E}),\Psi(A,\mathcal{E}))$.
    Recalling from the statement of Lemma~\ref{lem_coupling_two_MCs} that
    \[\sigma:= \inf\{t: \mathcal{B}_t \neq \Psi(\mathcal{A}_t)\}, \qquad T_a:= \inf\{t: f(\mathcal{A}_t) \ge a\}\]
    and abbreviating
    \[t_*:= \frac{2\log n}{|\log \varphi|},\qquad a_*:= \frac{1}{4 t_* \sqrt{n}} = \frac{|\log \varphi|}{8\sqrt{n} \log n},\]
    we bound the probability on the left-hand side of~\eqref{eq_tau_using_coupling} by
      \begin{align}
        &\P(\tau > t_*, \; \sigma > t_* )+ \P(\sigma \le t_* \wedge T_{a_*} ) + \P(T_{a_*} < \sigma \le t_* ).
    \end{align}
    
    By Lemma~\ref{lem_coupling_two_MCs}, we have
    \[\P(\sigma \le t_* \wedge T_{a_*}) \le 2t_*a_* = \frac{1}{2\sqrt{n}}.\]
    By the definition of~$\sigma$, on the event~$\{\tau > t_*,\; \sigma > t_*\}$ we have that~$\mathcal{B}_{t_*}$ is not empty. Then,~$\P(\tau > t_*,\; \sigma > t_*)$ is smaller than the probability that a herds process started with fewer than~$n^{1/6}$ particles is still alive by time~$t_*$. By~\eqref{eq_subadd0} and~\eqref{eq_const}, we obtain
    \begin{equation}
        \label{eq_will_use_later}
    \P(\tau > t_*,\; \sigma > t_*) \le Cn^{1/6} \cdot \varphi^{t_*} = Cn^{1/6} \cdot n^{-2} < \frac{1}{4\sqrt{n}}    
    \end{equation}
    if~$n$ is large enough.

    It remains to bound~$\P(T_{a_*} < \sigma \le t_*)$. Recalling that~$f(\mathcal{A}_t) = C_f(|\xi_t|+|\mathscr{E}_t|)^2/n$, if~$T_{a_*} < \infty$ we have
    \[ |\xi_{T_{a_*}}| + |\mathscr{E}_{T_{a_*}}| \ge (n a_*/C_f)^{1/2}. \]
    Since~$|\xi_0|+|\mathscr{E}_0| \le n^{1/6}$, we obtain that
    \[\text{on } \{T_{a_*} < \infty\}, \quad |\xi_{T_{a_*}}| + |\mathscr{E}_{T_{a_*}}| - (|\xi_0| + |\mathscr{E}_0|) \ge (n a_*/C_f)^{1/2} - n^{1/6}> n^{1/5}.\]
    Now, the process~$(|\xi_t|+|\mathscr{E}_t|)$ only changes at times when~$(\mathcal{A}_t) = (\xi_t,\mathscr{E}_t)$ changes. If~$\mathcal{A}_t$ has a new infection appearing at time~$t$, then~$|\xi_t|+|\mathscr{E}_t|$ may increase by at most 2 at that time. If, on the other hand,~$\mathcal{A}_t$ performs a jump of any other kind, then~$|\xi_t|+|\mathscr{E}_t|$ stays the same or decreases. Hence, for any~$t$, we have
    \[|\xi_t|+|\mathscr{E}_t| - (|\xi_0| + |\mathscr{E}_0|) \le 2 |\{s \le t:\; \xi_s = \xi_{s-}+1\}|.\]
    Putting these observations together, we see that
    \[\text{on } \{T_{a_*} < \sigma \le t_*\}, \quad  |\{s < \sigma:\; \xi_s = \xi_{s-}+1\}| \ge n^{1/5}/2. \]
    Moreover, before time~$\sigma$, whenever a new infection appears in~$(\mathcal{A}_t)$, a new particle is also born in~$(\mathcal{B}_t)$. Letting~$\mathcal{N}_\infty$ denote the number of particles ever born in~$(\mathcal{B}_t)$ (even after time~$\sigma$), we obtain the bound
    \[\P(T_{a_*} < \sigma \le t_*) \le \P(\mathcal{N}_\infty \ge n^{1/5}/2) \le \frac{\E[\mathcal{N}_\infty^p]}{(n^{1/5}/2)^p},\]
    for any~$p \ge 1$, by Markov's inequality. Recalling that~$\mathcal{B}_0$ has at most~$n^{1/6}$ particles, and using Lemma~\ref{lem_for_p_infinity} and Corollary~\ref{lem.sub.totsize}, the right-hand side is smaller than
    \[\frac{C_p (n^{1/6})^p}{(n^{1/5}/2)^p} = C_p' \cdot n^{-p/30},\]
for some constants~$C_p,C_p' > 0$. Taking~$p>15$ and then~$n$ large enough, this is smaller than~$\frac{1}{4\sqrt{n}}$, completing the proof.
\end{proof}

\begin{lemma}\label{lem_from_not_forest}
    There exist~$\varepsilon > 0$ and~$\delta > 0$ such that the following holds. Let~$(A,\mathcal{E}) \in \mathcal{X}_1 \backslash \mathcal{X}_1'$ (so that~$\mathrm{Graph}(A,\mathcal{E})$ is not a forest). Assume that~$|A| + |\mathcal{E}| \le n^\varepsilon$. Also assume that there is an edge~$e$ in~$\mathcal{E}$ such that~$(A,\mathcal{E}\backslash \{e\}) \in \mathcal{X}_1'$, that is,~$\mathrm{Graph}(A,\mathcal{E})$ would become a forest if~$e$ were removed from~$\mathcal{E}$.  Letting~$(\xi_t,\mathscr{E}_t)_{t \ge 0}$ denote the contact and exploration process started from~$(A,\mathcal{E})$, and letting~$\tau$ denote the extinction time of~$(\xi_t)$, for~$n$ large enough we have
    
\begin{equation}
    \label{eq_delta_and_phi}
\P\left( \tau >  \left( \delta + \frac{2}{|\log \varphi|} \right) \log n \right) < n^{-4\varepsilon}.
\end{equation}
\end{lemma}
\begin{proof}
Let~$U$ denote the time when the edge~$e$ disappears, due to being involved in a switch. The rate at which this happens equals~$\upupsilon_n = \frac{\mathsf{v}}{nd}$ times the number of other edges in the graph, which is~$\frac{nd}{2}-1$, times 2 (switches can be positive or negative). So this rate is~$\mathsf{v}(1 - \frac{\mathsf{1}}{2nd})$. Hence,~$U$ has exponential distribution with parameter~$\mathsf{v}(1 - \frac{\mathsf{1}}{2nd})$.

Denote by~$B$ the event that:
\begin{itemize}
    \item for all~$t \in [0,U)$ we have~$(\xi_t,\mathscr{E}_t \backslash \{e\}) \in \mathcal{X}_1'$, that is, it stays the case that the removal of~$e$ from the set of edges turns~$\mathrm{Graph}(\xi_t,\mathscr{E}_t)$ into a forest;
    \item $\mathrm{Graph}(\xi_U,\mathscr{E}_U)$ is a forest.
    \end{itemize}

We fix~$\delta > 0$ and~$\varepsilon > 0$ for now; their values will be chosen at the end of the proof.
We define
\[T:= \inf\{t:\; |\xi_t| + |\mathscr{E}_t| > n^{1/6}\}.\]
We bound the probability in~\eqref{eq_delta_and_phi} by
\begin{align}
    &\P(U > \delta \log n) \\
   \label{eq_second_term4} & + \P\left(U \le \delta \log n,\; T \le U\right)\\
    \label{eq_third_term4}& + \P\left(\left\{ U \le \delta \log n,\; T > U\right\} \cap B^c\right)\\
    & + \P\left(\left\{ U \le \delta \log n,\; T > U\right\} \cap B \cap \left\{\tau > \left(\delta + \frac{2}{|\log \varphi|} \right)\log n \right\}\right).\label{eq_fourth_term4}
\end{align}
We bound these four terms separately, starting with the first and last, which are the easiest. We have
\[\P(U > \delta \log n) = \exp\left\{ - \mathsf{v}\left(1-\frac{1}{2nd}\right) \delta \log n \right\} = n^{-\mathsf{v}\left(1-\frac{1}{2nd}\right) \delta}.\]
Next, letting~$B':=\left\{ U \le \delta \log n,\; T > U\right\} \cap B$, and letting~$(\mathcal{F}_t)_{t \ge 0}$ be the natural filtration of~$(\xi_t,\mathscr{E}_t)$, the probability in~\eqref{eq_fourth_term4} is
\begin{align*}
    &\P\left(B' \cap \left\{\tau > \left(\delta + \frac{2}{|\log \varphi|} \right)\log n \right\} \right) \\[.2cm]
    &= \mathbb{E}\left[ \mathds{1}_{B'} \cdot \E \left[ \left. \tau > \left(\delta + \frac{2}{|\log \varphi|} \right)\log n \; \right| \mathcal{F}_U\right]\right]\le \P(B') \cdot \frac{1}{\sqrt{n}} \le \frac{1}{\sqrt{n}},
\end{align*}
where the first inequality follows from Lemma~\ref{lem_stage_3}.

To bound~\eqref{eq_second_term4}, we note that~$|\xi_t|+|\mathscr{E}_t|$ can increase by at most two units at a given jump time, and a jump that causes such an increase happens with rate at most~$\lambda d|\xi_t|$. We can thus stochastically dominate~$(|\xi_t|+|\mathscr{E}_t|)_{t \ge 0}$ by a pure-birth process~$(Z_t)_{t \ge 0}$ on~$\mathbb{N}$ which starts from~$Z_0 = \lceil n^\varepsilon \rceil$ and jumps from~$k$ to~$k+2$ with rate~$\lambda d k$ (and has no other kind of jump). Then, the probability in~\eqref{eq_second_term4} is smaller than
\begin{equation*}\begin{split}
&\P\left( \max_{0 \le t \le \delta \log n} (|\xi_t| + |\mathscr{E}_t|) > n^{1/6}\right) \\
&\le \P(Z_{\delta \log n} > n^{1/6}) \le \frac{\mathbb{E}[Z_{\delta \log n}]}{n^{1/6}} = \frac{\lceil n^\varepsilon \rceil \cdot \exp\{2\lambda d \delta \log n\}}{n^{1/6}} \le 2n^{\varepsilon + 2\lambda d \delta - \frac16}.    
\end{split}\end{equation*}

We now turn to~\eqref{eq_third_term4}. For the process~$(\xi_t,\mathscr{E}_t)_{t \ge 0}$, let us say that a ``bad jump'' is a jump time when either (a) a contact process transmission occurs which causes the inclusion in the exploration process of an edge between two vertices that were already present in~$\mathrm{Graph}(\xi_t,\mathscr{E}_t)$, or (b) a switch involving two edges that were already in~$\mathscr{E}_t$. The point is that, as long as there are no bad jumps, it remains true that~$\mathrm{Graph}(\xi_t,\mathscr{E}_t)$ would become a forest if~$e$ were removed. In particular, letting~$S$ denote the time at which the first bad jump occurs, we have
\[
\P(\{U \le \delta \log n,\; T > U\} \cap B^c) \le \P(U \le \delta \log n,\; T > U,\; S \le U).
\]
Now let~$\mathcal{R}_t$ denote the rate at which a bad jump occurs from the present state~$(\xi_t,\mathscr{E}_t)$. It is straightforward to check that there is~$C>0$ such that~$\mathcal{R}_t \le C\frac{(|\xi_t| + |\mathscr{E}_t|)^2}{n}$, and that the process
\[Y_t:= \mathds{1}\{S \le t,\; S \le T\} - \int_0^{t \wedge S \wedge T} \mathcal{R}_s\;\mathrm{d}s,\quad t \ge 0\]
is a martingale. Then,
\[0 = \E[Y_0] = \E[Y_{\delta \log n}] \ge \P(S \le \delta \log n,\; S \le T) - \delta \log n \cdot C \frac{(n^{1/6})^2}{n}.\]
Then, we have
\[
\P(U \le \delta \log n,\; T > U,\; S \le U) \le \P(S \le \delta \log n,\; S \le T) \le C\delta \log n \cdot n^{-2/3}.
\]

Putting now all our bounds together, we have proved that the probability in~\eqref{eq_delta_and_phi} is smaller than
\[
n^{-\mathsf{v}\left(1-\frac{1}{2nd}\right) \delta} + n^{-1/2} + 2n^{\varepsilon + 2\lambda d \delta - \frac16} + C\delta \log n \cdot n^{-2/3}.
\]
By first choosing~$\delta$ small and then choosing~$\varepsilon$ much smaller, this expression is smaller than~$n^{-4\varepsilon}$ when~$n$ is large enough.
\end{proof}

\begin{proof}[Proof of Theorem~\ref{theo.2}]
Let~$\varepsilon$ and~$\delta$ be as in Lemma~\ref{lem_from_not_forest}. Define
\[\beta:= \frac{4}{|\log \varphi|} + \delta.\]
We will prove that
\[\P(\xi^{V_n}_{\beta \log n} \neq \varnothing) \xrightarrow{n \to \infty} 0.\]
By~\eqref{eq_apply_duality}, this will follow from proving that
\[\lim_{n \to \infty} n \cdot \P(\xi^{\bar{u}}_{\beta \log n} \neq \varnothing) = 0,\]
where~$\bar{u}$ is a deterministic vertex. In order to prove this, we take the coupling~$(\mathcal{A}_t,\mathcal{B}_t)_{t \ge 0}$ from Lemma~\ref{lem_coupling_two_MCs}, with~$\mathcal{A}_0 = (\xi_0,\mathscr{E}_0) = (\{\bar{u}\},\varnothing)$ and~$\mathcal{B}_0 = \Psi(\mathcal{A}_0)$. Recall that~$\sigma := \inf\{t: \mathcal{B}_t \neq \Psi(\mathcal{A}_t)\}$. 

Let~$\tau$ be the extinction time of~$(\xi_t)$ (which starts from~$\{\bar{u}\}$), and also define
\[\beta_0:= \frac{2}{|\log \varphi|}\]
and
\[T:= \inf\{t:\;|\xi_t|+|\mathscr{E}_t| \ge n^\varepsilon\}.\]
We bound:
\begin{align}
\nonumber    &\P(\tau > \beta \log n) \\
\label{eq_final_three1}&\le \P(\tau > \beta \log n,\; \sigma > \beta_0 \log n)\\
\label{eq_final_three2}&\quad + \P(\sigma \le \beta_0 \log n,\; T \le \sigma)\\
\label{eq_final_three3}&\quad + \P(\tau > \beta \log n,\; \sigma \le \beta_0 \log n,\; T > \sigma).
\end{align}

To bound~\eqref{eq_final_three1}, we give the same argument as we have used to bound~\eqref{eq_will_use_later}; here it gives:
\[
\P(\tau > \beta \log n,\; \sigma > \beta_0 \log n) \le C\varphi^{\beta_0 \log n} < n^{-2}.
\]

To bound~\eqref{eq_final_three2}, we observe again that~$|\xi_t| + |\mathscr{E}_t|$ can only increase by 2 at any given jump time, and it only increases when there are contact transmissions generating new births. Moreover, before time~$\sigma$, any time when there is a birth for~$(\xi_t,\mathscr{E}_t)$, there is also a particle birth for~$\mathscr{B}_t$. These considerations allow us to bound:
\begin{align*}
    &\P(\sigma \le \beta_0 \log n,\; T \le \sigma) \\
    &\le \P(\sigma \le \beta_0 \log n,\; |\{t \le \sigma:\; |\xi_t| = |\xi_{t-}| + 1\}| \ge n^\varepsilon/2)\\
    &\le \P\left(\sigma \le \beta_0 \log n,\; |\{t < \sigma:\; |\xi_t| = |\xi_{t-}| + 1\}| \ge \frac{n^\varepsilon}{2}-2\right)\\
    &\le \P\left((\mathscr{B}_t) \text{ has more than $\frac{n^\varepsilon}{2}-2$ births}\right).
\end{align*}
By Corollary~\ref{lem.sub.totsize}, this is smaller than
\[{C}_p\left( \frac{n^{\varepsilon}}{2}-2 \right)^{-p}\]
for any~$p \ge 1$. Hence, by taking~$p > 1/(2\varepsilon)$ and~$n$ large enough, it is smaller than~$1/n^2$.

We now turn to~\eqref{eq_final_three3}. Let us first bound, using Lemma~\ref{lem_coupling_two_MCs}:
\begin{equation}\label{eq_again_with_Cf}\P(\sigma \le \beta_0 \log n,\; T > \sigma) \le C_f \cdot \frac{(n^\varepsilon)^2}{n} \cdot \beta_0 \log n = C_f \beta_0 n^{-1+2\varepsilon} \log n.\end{equation}
Letting~$(\mathcal{F}_t)_{t \ge 0}$ be the natural filtration for~$(\mathcal{A}_t,\mathcal{B}_t)_{t \ge 0}$, we write
\begin{align}
\nonumber&\P(\tau > \beta \log n,\; \sigma \le \beta_0 \log n,\; T > \sigma) \\
\nonumber&\le \E\left[\mathds{1}\{\sigma \le \beta_0 \log n,\; T > \sigma\}\cdot \P(\tau > \beta \log n \mid \mathcal{F}_\sigma) \right] \\
 &=\E\left[\mathds{1}\{\sigma \le \beta_0 \log n,\; T > \sigma,\; \mathcal{A}_\sigma \notin \mathcal{X}_1' \}\cdot \P(\tau > \beta \log n \mid \mathcal{F}_\sigma) \right] \label{eq_final_final1}\\
&\quad +\E\left[\mathds{1}\{\sigma \le \beta_0 \log n,\; T > \sigma,\;\mathcal{A}_\sigma \in \mathcal{X}_1'\}\cdot \P(\tau > \beta \log n \mid \mathcal{F}_\sigma) \right].\label{eq_final_final2}
\end{align}
To bound~\eqref{eq_final_final1}, we note that on the event~$\{\sigma \le \beta_0 \log n,\; T > \sigma,\; \mathcal{A}_\sigma \notin \mathcal{X}_1'\}$, we have that~$\mathcal{A}_\sigma = (\xi_\sigma,\mathscr{E}_\sigma)$ satisfies the assumptions of Lemma~\ref{lem_from_not_forest}, and then that lemma implies that, on this event,~$\P(\tau > \beta \log n \mid \mathcal{F}_\sigma) \le n^{-4\varepsilon}$. Together with~\eqref{eq_again_with_Cf}, this implies that~\eqref{eq_final_final1} is smaller than
\[C_f \beta_0 n^{-1+2\varepsilon} \log n \cdot n^{-4\varepsilon} < n^{-1-\varepsilon}\]
if~$n$ is large enough.

    Next, Lemma~\ref{lem_stage_3} implies that on~$\{\sigma \le \beta_0 \log n,\; T > \sigma,\;\mathcal{A}_\sigma \in \mathcal{X}_1'\}$, we have~$\P(\tau > \beta \log n \mid \mathcal{F}_\sigma) < n^{-1/2}$; combining this with~\eqref{eq_again_with_Cf} shows that~\eqref{eq_final_final2} is smaller than  
    \[C_f \beta_0 n^{-1+2\varepsilon} \log n \cdot n^{-1/2} < n^{-5/4}\]
for~$n$ large. This completes the proof.
    \end{proof}

\section{Appendix}
\subsection{Proofs of Lemma~\ref{lem_integral_g} and Lemma~\ref{lem_derivative_A}}

\begin{proof}[Proof of Lemma~\ref{lem_integral_g}]
	The proof is the same for the two functions, so we only treat the first. We use the simple bounds, that come from comparison with a pure birth process,
	\begin{align*}&\mathbb{E}[X_t \mid \Xi_0 = \delta_A] \le |A|\cdot e^{d\lambda t}, \quad \mathbb{E}[X_t \mid \Xi_0 = \delta_{A_{e,1}} + \delta_{A_{e,2}}] \le |A|\cdot e^{d\lambda t},\end{align*}
together with the expression~\eqref{eq_def_g}, to obtain
	\begin{align*}
		g_\mathsf{v}(\xi,t) &\le e^{d\lambda t} \cdot \sum_A \xi(A)\cdot |A|\cdot |\{\text{active edges of $A$}\}|\\
	&\le e^{d\lambda t} \cdot \left(\sum_A \xi(A)\cdot |A|\right) \left(\sum_A \xi(A)\cdot |\{\text{active edges of $A$}\}|\right)\\ &\le e^{d\lambda t} \cdot (X(\xi) + \mathscr{E}(\xi))^2.
	\end{align*}
	The statement now readily follows from Corollary~\ref{cor_explosion}.
\end{proof}

The following preliminary result will allow us to perform the exchange of limit and expectation in~\eqref{eq_exchange}.

\begin{lemma}
	\label{lem_bound_dct} There exist~$c_1,\;c_2> 0$ (depending on~$\lambda,\mathsf{v},\varepsilon)$ such that for any~$0 \le t < t+s <T$, on~$\{\tau_{\mathrm{sep}} > t\}$ we have
	\begin{equation*}
		|\widehat{\E}[\mathcal{A}_{t+s} \mid \mathcal{F}_t]| \le c_1(X(\mathcal{V}_t) + \mathscr{E}(\mathcal{V}_t))^{c_2}\cdot s.
	\end{equation*}
\end{lemma}
\begin{proof}
	We fix~$s,t$ as in the statement, and let~$N$ denote the number of jumps of the process~$(\mathcal{V},\mathcal{W})$ in~$[t,t+s]$. We have
	\begin{align}
		\nonumber &\mathds{1}_{\{\tau_\mathrm{sep} > t\}}\cdot |\widehat{\E}\left[\mathcal{A}_{t+s} \mid \mathcal{F}_t\right]|\\
		\label{eq_first_withN}&\le \mathds{1}_{\{\tau_\mathrm{sep} > t\}}\cdot \left|\widehat{\E}\left[\left.\mathds{1}_{\{\tau_\mathrm{sep} \le t+s,\; N = 1\}}\cdot \widehat{\E}[\mathcal{Y}_T - \mathcal{X}_T \mid \mathcal{F}_{\tau_\mathrm{sep}}] \;\right|\; \mathcal{F}_t\right]\right| \\
		\label{eq_second_withN}&\quad+\mathds{1}_{\{\tau_\mathrm{sep} > t\}}\cdot \left|\widehat{\E}\left[\left.\mathds{1}_{\{\tau_\mathrm{sep} \le t+s,\; N \ge 2\}}\cdot \widehat{\E}[\mathcal{Y}_T - \mathcal{X}_T \mid \mathcal{F}_{\tau_\mathrm{sep}}] \;\right|\;\mathcal{F}_t\right]\right|.
	\end{align}
	We will treat the terms~\eqref{eq_first_withN} and~\eqref{eq_second_withN} separately.
For both, it will be useful to bound, using domination by a pure birth process,
	\[\widehat{\E}[\mathcal{X}_T \mid \mathcal{F}_{\tau_\mathrm{sep}}] \le e^{d\lambda T} X(\mathcal{V}_{\tau_\mathrm{sep}}),\quad \widehat{\E}[\mathcal{Y}_T \mid \mathcal{F}_{\tau_\mathrm{sep}}] \le e^{d\lambda T} X(\mathcal{W}_{\tau_\mathrm{sep}})= e^{d\lambda T} X(\mathcal{V}_{\tau_\mathrm{sep}}),\]
	which gives
	\begin{equation}\label{eq_for_two_bounds}
	|\widehat{\E}[\mathcal{Y}_T - \mathcal{X}_T \mid \mathcal{F}_{\tau_\mathrm{sep}}]| \le e^{d\lambda T}X(\mathcal{V}_{\tau_\mathrm{sep}}).
	\end{equation}
	
	We start bounding~\eqref{eq_first_withN}. On the event~$\{t < \tau_\mathrm{sep} \le t+s,\; N = 1\}$, we have~$X(\mathcal{V}_t) = X(\mathcal{W}_t) =X(\mathcal{V}_{\tau_\mathrm{sep}}) =  X(\mathcal{W}_{\tau_\mathrm{sep}})$ (since in this event the only jump of~$(\mathcal{V},\mathcal{W})$ in~$[t,t+s]$ is a split which causes the separation of the processes, so there is no change in the number of particles). Then, using~\eqref{eq_for_two_bounds}, we obtain that~\eqref{eq_first_withN} is smaller than
	\begin{align*}
	\mathds{1}_{\{\tau_\mathrm{sep} > t\}}\cdot  e^{d\lambda T} X(\mathcal{V}_t) \cdot \widehat{\P}( \tau_\mathrm{sep} \le t+s,\; N\ge 1 \mid \mathcal{F}_t ).
	\end{align*}
	The rate with which the process jumps away from the state~$(\mathcal{V}_t,\mathcal{W}_t)$ is at most
	\[\mu_1(\mathcal{V}_t):= (d\lambda+1) X(\mathcal{V}_t) + (\mathsf{v}+\varepsilon)\mathscr{E}(\mathcal{V}_t),\]
so the probability above is at most
	\[1-\exp\{-\mu_1(\mathcal{V}_t)\cdot s\}\le \mu_1(\mathcal{V}_t)\cdot s. \]
	We have thus proved that~\eqref{eq_first_withN} is bounded by the desired expression.

	We now turn to~\eqref{eq_second_withN}. We start using~\eqref{eq_for_two_bounds} to bound~\eqref{eq_second_withN} by
	\begin{align*}
		&\mathds{1}_{\{\tau_\mathrm{sep} > t\}}\cdot e^{d\lambda T}\cdot  \widehat{\E}\left[\left.\mathds{1}_{\{\tau_\mathrm{sep} \le t+s,\; N \ge 2\}}\cdot X(\mathcal{V}_{\tau_\mathrm{sep}})  \;\right|\;\mathcal{F}_t\right]\\[.2cm]
		&\le \mathds{1}_{\{\tau_\mathrm{sep} > t\}}\cdot e^{d\lambda T}\cdot \widehat{\E}\left[\left.\mathds{1}_{\{ N \ge 2\}}\cdot \max_{u \in [t,T]} X(\mathcal{V}_u)  \;\right|\;\mathcal{F}_t\right].
	\end{align*}
	By H\"older's inequality, the expectation on the right-hand side is smaller than
	\begin{align}\label{eq_after_holder}
		\mathds{1}_{\{\tau_\mathrm{sep} > t\}}\cdot e^{d\lambda T}\cdot	\widehat{\P}\left(N \ge 2 \mid \mathcal{F}_t\right)^{2/3} \cdot \widehat{\E}\left[\left. \max_{u \in [t,T]}  X(\mathcal{V}_u)^3   \;\right|\;\mathcal{F}_t\right]^{1/3}.
	\end{align}
	Corollary~\ref{cor_explosion} and the Markov property imply that
	\begin{equation}\label{eq_just_corollary}
	\widehat{\E}\left[\left. \max_{u \in [t,T]} X(\mathcal{V}_u)^3   \;\right|\;\mathcal{F}_t\right]^{1/3} \le cX(\mathcal{V}_t). 
	\end{equation}
	Let
	\[\mu_2(\mathcal{V}_t):= 2(d\lambda+1)(X(\mathcal{V}_t)+1) + 2(\mathsf{v}+\varepsilon)(\mathscr{E}(\mathcal{V}_t)+1).\]
	In the event~$\{\tau_\mathrm{sep} > t\}$, the process~$(\mathcal{V},\mathcal{W})$ jumps away from~$(\mathcal{V}_t,\mathcal{W}_t)$ with rate smaller than~$\mu_1(\mathcal{V}_t)$ (as previously observed), and after performing a first jump, it performs a second jump with a rate that is smaller than~$\mu_2(\mathcal{V}_t)$. Indeed, after the first jump the number of particles or active edges of~$\mathcal{V}$ and~$\mathcal{W}$ increase by at most 1, and the factors~$2$ in the definition of~$\mu_2(\mathcal{V}_t)$ account for the possibility that the first jump is the separation of the two processes.
	Letting~$({Z}_t)_{t \ge 0}$ be a Poisson process with constant rate~$\mu_2(\mathcal{V}_t)$, we bound (on the event~$\{\tau_\mathrm{sep} > t\}$):
	\begin{align*}
		\widehat{\P}(N \ge 2 \mid \mathcal{F}_t) \le \P({Z}_s \ge 2) &= 1-\mathrm{e}^{-\mu_2(\mathcal{V}_t) s} - \mu_2(\mathcal{V}_t) s\cdot  \mathrm{e}^{-\mu_2(\mathcal{V}_t) s} \\&\le \mu_2(\mathcal{V}_t) s - \mu_2(\mathcal{V}_t) s \cdot \mathrm{e}^{-\mu_2(\mathcal{V}_t) s} \\
		&\le (\mu_2(\mathcal{V}_t) s)^2.
	\end{align*}
	Plugging this bound and~\eqref{eq_just_corollary} back in~\eqref{eq_after_holder} gives the desired bound.
\end{proof}

\begin{proof}[Proof of Lemma~\ref{lem_derivative_A}]
	Fix~$t \in [0,T)$. For any~$s \in (0,T-t]$, we have
	\[\mathds{1}\{\tau_\mathrm{sep} > t\} \cdot \frac{\widehat{\E}[\mathcal{A}_{t+s} \mid \mathcal{F}_t]}{s} \le \mathds{1}\{\tau_\mathrm{sep} > t\} \cdot c_1(X(\mathcal{V}_t) + \mathscr{E}(\mathcal{V}_t))^{c_2}.\]
	The random variable on the right-hand side is integrable, by Corollary~\ref{cor_explosion}. This and the Dominated Convergence Theorem justify the exchange of limit in~\eqref{eq_exchange}. As explained before~\eqref{eq_exchange}, this implies that
	\begin{align*}
		\lim_{s \to 0+} \frac{\widehat{\E}[\mathcal{A}_{t+s}] - \widehat{\E}[\mathcal{A}_t]}{s} &= \widehat{\E}\left[\mathds{1}\{\tau_{\mathrm{sep}} > t\} \cdot \lim_{s \to 0+} \frac{\widehat{\E}[\mathcal{A}_{t+s}\mid \mathcal{F}_t]}{s}\right]\\[.2cm]
		&= \varepsilon \cdot \widehat{\E}[\mathds{1}\{\tau_{\mathrm{sep}} > t\}\cdot g_{\mathsf{v},\varepsilon}(\mathcal{V}_t, T-t)].
	\end{align*}
	Finally, any function~$g:[0,\infty) \to \R$ which is continuous and has a continuous derivative from the right is necessarily differentiable (with derivative equal to the derivative from the right), so the proof is complete.
\end{proof}

\renewcommand{\baselinestretch}{1}
\setlength{\parskip}{0pt}
\small

\bibliography{refs}
\bibliographystyle{alpha}

\end{document}